\documentclass[11pt,reqno]{amsart}

\usepackage{amsmath}
\usepackage{amsthm}
\usepackage[all]{xy}
\usepackage{color}
\usepackage{amssymb}
\usepackage{graphicx}
\theoremstyle{plain}
\newtheorem{theorem}{Theorem}[section]
\newtheorem{proposition}[theorem]{Proposition}
\newtheorem{lemma}[theorem]{Lemma}
\newtheorem{corollary}[theorem]{Corollary}

\newtheorem{question}[theorem]{Question}

\newtheorem{definition}[theorem]{Definition}

\def\int{\text{interior}}
\def\hyp {\hbox {\rm {H \kern -2.8ex I}\kern 1.25ex}}
\def\reals {\hbox {\rm {R \kern -2.8ex I}\kern 1.15ex}}
\def\integers {\hbox {\rm { Z \kern -2.8ex Z}\kern 1.15ex}}
\def\naturals {\hbox {\rm {N \kern -2.8ex I}\kern 1.20ex}}
\def\rationals {\hbox {\rm { Q \kern -2.2ex l}\kern 1.15ex}}
\def\hyp {\hbox {\rm {H \kern -2.7ex I}\kern 1.25ex}}
\def\bar{\overline}

\newcommand{\AAA}{\mathcal{A}}
\newcommand{\DD}{\mathcal{D}}
\newcommand{\EE}{\mathcal{E}}
\newcommand{\R}{\mathbb{R}}
\newcommand{\C}{\mathbb{C}}
\newcommand{\Z}{\mathbb{Z}}

\newcommand{\TT}{\mathcal{T}}

\newcommand{\NN}{\mathcal{N}}
\newcommand{\nbhd}{\operatorname{\NN}}
\newcommand{\Hyp}{\mathbb{H}}
\newcommand{\boundary}{\partial}
\newcommand{\union}{\cup}
\newcommand{\intersect}{\cap}
\newcommand{\ssm}{\smallsetminus}
\newcommand{\hhat}{\widehat}

\catcode `\@=11
\long\def\@savemarbox#1#2{\global\setbox#1\vtop{\hsize\marginparwidth 
  \@parboxrestore\tiny\raggedright #2}}
\catcode`\@=12

\def\strutdepth{\dp\strutbox}
\def \ss{\strut\vadjust{\kern-\strutdepth \sss}}
\def \sss{\vtop to \strutdepth{
\baselineskip\strutdepth\vss\llap{$\diamondsuit\;\;$}\null}}

\def\strutdepth{\dp\strutbox}
\def \sst{\strut\vadjust{\kern-\strutdepth \ssss}}
\def \ssss{\vtop to \strutdepth{
\baselineskip\strutdepth\vss\llap{$\spadesuit\;\;$}\null}}

\def\strutdepth{\dp\strutbox}
\def \ssh{\strut\vadjust{\kern-\strutdepth \sssh}}
\def \sssh{\vtop to \strutdepth{
\baselineskip\strutdepth\vss\llap{$\heartsuit\;\;$}\null}}

\title[Discrete primitive-stable representations]{Discrete primitive-stable representations with large rank surplus}

\date{\today}

\author{Yair Minsky}
\address{\hskip-\parindent
        Department of Mathematics \\
        Yale University \\
        PO Box 208283 \\
        New Haven, CT 06520 \\
        USA}
\email{yair.minsky@yale.edu}

\author{Yoav Moriah}
\address{\hskip-\parindent
        Department of Mathematics \\
        Technion \\
        Israel}
\email{ymoriah@techunix.technion.ac.il}


\thanks{Research partially  supported by  NSF grant DMS-0504019 and  BSF grant  2011256.
The second author would like to thank Yale University for its hospitality during
 a sabbatical in which the research was initiated, and the first
 author would like to thank the Technion for its hospitality during a
 visit in which it was completed.}

\begin{document}

\begin{abstract}
We construct a sequence of primitive-stable representations of free
groups into $PSL_2(\C)$ whose ranks go to infinity, but whose images
are discrete with quotient manifolds that converge geometrically to a
knot complement. In particular this implies that the rank and geometry
of the image of a primitive-stable representation imposes no constraint on the rank of
the domain. 
\end{abstract}

\maketitle

\section{Introduction} \label{introduction}
\newcommand\cv{{\mathcal X}}

\vskip20pt

Let $F_n$ denote the free group on $n$ generators, where $n\ge 2$. 
The space $Hom(F_n,PSL_2(\C))$ of representations of $F_n$ into
$PSL(2,\C)$ contains within it presentations of all hyperbolic 3-manifold
groups of rank bounded by $n$, and so is of central interest in
three-dimensional geometry and topology. On the other hand there is
also an interesting dynamical structure on $Hom(F_n,PSL_2(\C))$ coming
from the action of $Aut(F_n)$ by precomposition (see Lubotzky
\cite{lubotzky:autFn}). The interaction between the geometric and
dynamical aspects of this picture is still somewhat mysterious, and
forms the motivation for this paper. 

(Note that it is natural to identify representations conjugate in $PSL_2(\C)$, so that
in fact we often think about the {\em character variety} $\cv(F_n)$
and the natural action by $Out(F_n)$, the outer automorphism group of
$F_n$. This distinction will not be of great importance here.)

In \cite{minsky:primitive} the notion of a {\em primitive-stable}
representation $\rho:F_n \to PSL_2(\C)$ was introduced. The set
$PS(F_n)\subset \cv(F_n)$ of
primitive stable conjugacy classes is open and
contains all Schottky representations (discrete, faithful
representations with compact convex core), but it also contains
representations with dense image, and nevertheless $Out(F_n)$ acts
properly discontinuously on $PS(F_n)$.  This implies, for example, that
$Out(F_n)$ does not act ergodically on the (conjugacy classes of)
representations with dense image.

Representations into $PSL_2(\C)$, whose images are discrete, torsion-free 
subgroups, give rise to hyperbolic 3-manifolds, and when the volume
of the 3-manifold is finite we know by Mostow-Prasad rigidity that the
representation depends uniquely, up to conjugacy, on the presentation
of the abstract fundamental group. Hence it makes sense to ask whether
a presentation of such a 3-manifold group is or is not
primitive-stable. 

It is not hard to show that primitive-stable presentations of closed
3-manifold groups do exist, and such presentations are constructed in
this paper, but we are moreover concerned with the relationship
between the rank of the presentation and the rank of the group. 

Our goal will be to show that the rank of the presentation can in fact
be arbitrarily higher than the rank of the group, and more specifically:

\begin{theorem}\label{maintheorem} 
There is an infinite sequence of  representations $\rho^r: F_{n_r} \to PSL_2(\C)$, where
$n_{r} = n_0 + 2r$,  so that :

\begin{enumerate}

\item Each $\rho^r$ has discrete and torsion-free image. 

\smallskip

\item Each $\rho^r$ is primitive-stable. 

\smallskip

\item The quotient manifolds $N_r = \Hyp^3/\rho^r(F_{n_r})$ converge geometrically to $N_{\infty}$, 
where $N_{\infty}$ is a knot complement in $S^3$. 

\end{enumerate}

\end{theorem}

In particular note that, because the quotient manifolds converge
geometrically to a fixed finite volume limit,  the rank as well as the
covolume of the image groups remains bounded while $n_r\to\infty$ (see
e.g. Thurston \cite{wpt:notes}), hence: 

\begin{corollary} \label{arbitraryhighrank}
There exists $R$ such that, for each $n$, there is a lattice in
$PSL_2(\C)$ with rank bounded by $R$, which is the image of a
primitive stable representation of rank greater than $n$.  
\end{corollary}

As the reader might guess our construction involves a sequence of Dehn
fillings of a knot  complement, and in particular the 
manifolds $N_r$ are in infinitely many homemorphism types. Thus we are
currently unable to answer the following natural question:

\begin{question}
 Is there a single lattice  $G \subset PSL_2(\mathbb{C})$ which has
 primitive stable presentations of arbitrarily high rank? 
\end{question}

To put this in context we note that  (as follows from the results in
\cite{minsky:primitive}) simply adding generators to a representation
which map to the group generated by the previous generators
immediately spoils the property of primitive stability. Thus the
existence of primitive stable presentations is
delicate to arrange.

On the other hand, we do not have examples in the other direction
either: 

\begin{question}
Are there any lattices in $PSL_2(\C)$ which do not have {\em any}
primitive stable presentations? 
\end{question}

The only tool we have for proving primitive-stability involves
Heegaard splittings which must satisfy a number of interacting
conditions. It would be interesting to know if this is always the
case: 

\begin{question}
Is every primitive stable presentation of a closed hyperbolic
3-manifold group {\em geometric}, i.e. does it arise from one side of
a Heegaard splitting?
\end{question}

\subsection*{Outline of the construction}

Our starting point is a class of knots supported on surfaces in $S^3$
in a configuration known as a {\em trellis}, as previously studied by
Lustig-Moriah \cite{lustig-moriah:highgenus}. The surface $\Sigma$ on
which such a knot $K$ is supported splits $S^3$ into two
handlebodies. For appropriately chosen special cases we find that the
complement $S^3\ssm K$ is hyperbolic, and that the representation
obtained from one of the handlebodies is primitive stable. Most of the
work for this is done in Section \ref{SCPS}, about which we remark more
below. 

To our chosen examples we can apply {\em flype moves} 
(as used by Casson-Gordon, see Moriah-Schultens \cite[Appendix]{moriah-schultens}),
which are isotopies of the knot that produce new trellis
projections, with higher genus. We show that these new
projections still yield primitive stable representations. 

Hence our knot complement $S^3\ssm K$  admits a sequence of homomorphisms
$\rho^r_{\infty}:F_{n_r} \to \pi_1(S^3\ssm K)$ with ranks $n_r \to\infty$, all
of which are primitive stable. However, these maps are not
surjective. 

To address this issue  we perform Dehn fillings on
$S^3\ssm K$, obtaining closed manifolds equipped with surjective
homomorphisms from $F_{n_r}$. Thurston's Dehn Filling Theorem tells us that, fixing
the flype index $r$ and letting the Dehn filling coefficient go to
infinity, we eventually obtain hyperbolic manifolds, and the
corresponding representations $\rho^r_m$ converge to $\rho^r_{\infty}$. Since primitive
stability is an open condition we eventually obtain our desired
primitive stable presentations. 

\medskip

Section 2 provides a little bit of background on hyperbolic
3-manifolds. In Section 3, we discuss primitive stability and prove
Proposition \ref{mainproposition}, which gives topological conditions
for primitive stability of a representation arising from a Heegaard
splitting where a knot on the Heegaard surface has been deleted. The
proof of this is an application of Thurston's covering theorem, and
of the main result of \cite{minsky:primitive}.

In Section \ref{KCT} we introduce trellises and our notation for knots
carried on them, recall a theorem from Lustig-Moriah \cite{lustig-moriah:highgenus}, and discuss horizontal surgeries.

In Section \ref{SCPS} we show that, under appropriate assumptions, a knot
carried by a trellis satisfies the conditions of Proposition
\ref{mainproposition}, and moreover the same is true for the configurations obtained
by flype moves on this knot.
Theorem \ref{hyperbolic} establishes that the knot complements we work
with are hyperbolic.  Intuitively one expects that complicated
diagrams such as we are using should ``generically'' yield hyperbolic
knots, but the proof turns out to be somewhat long and
painful. We perform a case-by-case analysis of the features of the
knot diagram, which is complicated by various edge effects in the
trellis. This analysis shows  that the manifold has no essential
tori, and  the same techniques also apply, in Proposition
\ref{flypedtrellis}, to show that the exterior pared handlebody
determined by a flyped trellis is never an $I$-bundle, which is
also one of the conditions needed in Proposition
\ref{mainproposition}.

The level of generality we chose for our family of examples, for
better or worse, is restricted enough to simplify some of the
arguments in Section \ref{SCPS}, but still broad enough to allow a
wide variation. It is fairly clear that the construction should work for an
even wider class of examples, but satisfying primitive stability, hyperbolicity, 
as well as the no-$I$-bundle condition is tricky and the resulting complication of our 
arguments would have diminishing returns for us and our readers. 

\vskip15pt

\section{Cores and ends of hyperbolic manifolds}\label{background}

\vskip10pt

In this section we review the basic structure of hyperbolic
3-manifolds and their ends. This will be applied in Section \ref{PS reps}.

A {\em compact core} of a 3-manifold $N$ is a compact submanifold $C$ of $N$ whose 
inclusion is a homotopy-equivalence. Scott \cite{scott:core} showed that every
irreducible 3-manifold with finitely generated fundamental group has a compact core.

Let $N = \Hyp^3/\Gamma$ be an oriented hyperbolic 3-manifold where $\Gamma$ is a
discrete torsion-free subgroup of $PSL_2(\C)$, and let $N^0$
denote $N$ minus its standard (open) cusp neighborhoods. Each cusp
neighborhood is associated to a conjugacy class of maximal parabolic
subgroups of $\Gamma$, and its boundary is an open annulus or a torus. 
For each component $T$ of $\boundary N^0$ let $T'$ be an
essential compact subannulus when $T$ is an annulus, and let $T'=T$ if
$T$ is a torus.

The relative compact core theorem of  McCullough \cite{mccullough:relative} and 
Kulkarni-Shalen \cite{kulkarni-shalen} implies:

\begin{theorem}\label{relative core}
There is a compact core $C\subset N^0$, such that $\boundary C \intersect T = T'$ for 
every component of $\boundary N^0$.

\end{theorem}

We call $C$ a {\em relative compact core}, and call $P=\boundary C \intersect
\boundary N^0$ the {\em parabolic locus} on its boundary. The pair
$(C,P)$ is called a {\em pared manifold} (see Morgan \cite{morgan:geometrization}).

Suppose that the components of $\partial C \ssm P$ are incompressible. Then Bonahon showed 
in   \cite{bonahon} that the components of $N^0 \ssm C$  are in one-to-one correspondence
with the components of $\boundary C \ssm P$, and each of them is a neighborhood of a unique end of $N^0$. 
Note that $C$ can be varied by isotopy and by choice of the annuli $T'$, so that an end can have many
neighborhoods. 

We say that an end of $N^0$ is {\em geometrically finite} if it has a
neighborhood which is entirely outside of the convex core of $N$ 
(where the convex core of $N$ is the smallest closed convex subset of $N$
whose inclusion is a homotopy equivalence).

Bonahon's tameness theorem \cite{bonahon} shows that every end of $N^0$ is either
geometrically finite or {\em simply degenerate}.  We will not need the definition of  this 
property, but will use the fact that it has  Thurston's Covering Theorem as a consequence. 
The Covering Theorem will be described and used in the proof of Proposition \ref{mainproposition}.

We remark that something similar to all this holds when $\partial C \ssm P $ is compressible
via the solution to the Tameness Conjecture by Agol \cite{agol:tame} and Calegari-Gabai \cite{calegari-gabai:tame}, 
but we will not need to use this.

\vskip20pt

\section{Primitive stable representations}
\label{PS reps}

\vskip10pt

In this section we summarize notation and facts from
\cite{minsky:primitive}, and prove Proposition \ref{mainproposition}, which gives a
sufficient condition for certain representations arising from knot
complements to be primitive stable.

\vskip0pt

Fix a generating set $\{X_1,\ldots,X_n\}$ of $F_n$ and  let  $\Gamma$ be a bouquet of oriented circles labeled by  
the $X_i$.  We let $\mathcal{B} = \mathcal{B}(\Gamma)$ denote the set of bi-infinite (oriented) geodesics in  
$\Gamma$. Each such geodesic lifts to an $F_n$-invariant set of bi-infinite geodesics in the universal covering 
tree $\tilde{\Gamma}$. The set $\mathcal{B} $ admits a natural action by   $Out(F_n)$.

An element of $F_n$ is called {\em primitive} if it is a member of a free generating set, or 
equivalently  if  it is the image of a  generator $X_i$ by an element of $Aut(F_n)$. Let 
$\mathcal{P} =\mathcal{P} (F_n)$ denote the subset of $\mathcal{B} $ consisting of geodesic
representatives of conjugacy classes of primitive elements.  Note that $\mathcal{P} 
$ is $Out(F_n)$-invariant.  

Given a representation $\rho:F_n\to PSL_{2}(C)$ and a basepoint $x\in\mathbb{H}^3$, there is 
a unique map  $\tau_{\rho,x} : \tilde{ \Gamma} \to\mathbb{H}^3$ mapping a selected vertex of 
$\tilde{\Gamma}$ to $x$, $\rho$-equivariant, and mapping each edge to a geodesic. 

\begin{definition}\label{ps def}\rm
A representation $\rho:F_n\to PSL_{2}(C)$ is {\em primitive-stable} if there are constants 
$K,\delta \in \mathbb{R}$  and a basepoint $x\in\mathbb{H}^3$ such that $\tau_{\rho,x}$ takes 
the leaves of  $\mathcal{P}$ to $(K,\delta)$-quasi geodesics in $\mathbb{H}^3$. 

\end{definition}

The property of primitive stability of a representation  is invariant under conjugacy in $PSL_{2}(C)$ and action 
by $Aut(F_n)$.  We define $\mathcal{PS}(F_n)$ to be the set of  (conjugacy classes of) primitive-stable representations.

It is easy to see that Schottky representations are primitive stable;
indeed, the Schottky condition is equivalent to saying that the map
$\tau_{\rho,x}$ is a quasi-isometric embedding on the entire tree at  
once.

\begin{definition}\label{Whiteheadgraph}\rm 
Given a free group $F(X) = F(X_1, \dots, X_n)$ on $n$ generators and a
cyclically reduced word $w = w(X_1, \dots, X_n)$ the {\it Whitehead
  graph} of $w$ with respect to the generating set $X$, denoted 
$Wh(w,X)$, is defined as follows: The vertex set of the graph consists
of $2n$ points labeled by the elements of
$X^\pm \equiv \{X_i^{\pm 1}\}$.  For each sub-word $UV$ in $w$ or its
cyclic permutations, where $U,V\in X^\pm$, we place an edge between
the points $U$ and $V^{-1}$.

\end{definition}

\begin{definition}\label{cutpointfree}\rm We say that a graph $\Gamma$ is {\it cut point free} if it is connected 
and contains no point $p \in \Gamma$ so that $\Gamma \smallsetminus p$ is not connected.

\end{definition}

\vskip10pt

It is a theorem of
Whitehead~\cite{whitehead:graph,whitehead:equivalence} that if $u \in
F_n$ is a cyclically reduced primitive word 
then $Wh(u, X)$ is {\it not} cut point free.

If $H$ is a handlebody and $\gamma$ is a curve in $\boundary H$, for a generating system $X$ 
for $\pi_1(H)$ let $Wh(\gamma,X)$ denote $Wh([\gamma],X)$, where $[\gamma]$ is a reduced 
word in $X^\pm$ representing $\gamma$ in $\pi_1(H)$. 

\medskip

The main result of \cite{minsky:primitive} states:

\begin{theorem}\label{blocking PS} {\rm [Theorem 4.1 of \cite{minsky:primitive}]}
If $\rho:\pi_1(H) \to PSL_2(\C)$ is discrete, faithful and geometrically finite, with a single cusp 
$c$ such that $Wh(c,X)$ is cut-point free for some set of generators $X$ of $\pi_1(H)$, then
$\rho$ is primitive stable. 

\end{theorem}

\vskip10pt

This theorem will allow us to prove the following proposition, which in turn will be a step in the proof of 
Theorem \ref{maintheorem}. Here and in the rest of the paper, we let
$\bar\NN_X(Y)$ denote a closed regular neighborhood of $Y$ in $X$, and $\NN_X(Y)$ its
interior. If the ambient space $X$ is understood we abbreviate to  $\NN(Y)$.
When we say a manifold with boundary is hyperbolic we mean its
interior admits a complete hyperbolic structure. 

\vskip10pt

\begin{proposition} \label{mainproposition}
Let $M$ be a closed $3$-manifold with a Heegaard splitting  $M = H_1 \cup_{\Sigma} H_2$, 
where $\Sigma = \partial H_1 = \partial H_2$. Let  $\gamma \subset \Sigma$ be a simple closed 
curve so that $M_\infty = M \smallsetminus \nbhd(\gamma)$ is a  hyperbolic manifold and:

\begin{enumerate}

\item The group $\pi_1(H_1)$ has a generating set $x  = \{x_1, \dots. x_n\}$ so that the Whitehead graph
$Wh(\gamma, x)$ is cut point free.

\vskip5pt

\item The subsurface $\overline{\Sigma \smallsetminus \nbhd(\gamma)}$ is incompressible 
in $M_\infty$. 

\vskip5pt

\item 
The pared manifold  $(H_2,\bar\NN_{\Sigma}(\gamma))$ is not an $I$-bundle.

\end{enumerate}

\vskip5pt

Let $ \widehat{H_1} = H_1 - \nbhd(\gamma)$ and let  $\,\, i^0 : \widehat{H_1} \to M_\infty$ be the map induced 
by the inclusion of $H_1 \to M$ and $\eta:\pi_1( M_\infty) \to
PSL_2(\mathbb{C})$ be  
a holonomy representation for the hyperbolic structure on $int(M_\infty)$.
Then the representation   $\rho = \eta \,\circ\, i^0_{*}$ given by
$$\pi_1(\widehat H_1) \overset{i^0_{*}}{\to} \pi_1(M_\infty) \overset{\eta}{\to} PSL_2(\mathbb{C})$$ 
is primitive stable.

\end{proposition}

Since $\widehat H_1$ is a deformation retract of $H_1$ we can naturally identify
$\pi_1(H_1)$ with $\pi_1(\widehat H_1)$. 

\medskip

Recall that in  a pared manifold $(M,P)$ (where $P\subset \partial M$ is a union of annuli and tori),
an {\it essential annulus} is  a properly embedded $\pi_1$ injective annulus in $(M,\partial M\smallsetminus P)$ 
which is  not properly homotopic into $P$ or into $\partial M\smallsetminus P$.
We say that $(M,P)$ is {\em acylindrical} if it contains no  essential annuli. 

\vskip6pt

We will need the following lemma:

\vskip6pt

\begin{lemma}\label{uniqueelement} 
Let $M$ be a closed $3$-manifold with a Heegaard splitting  $M = H_1 \cup_{\Sigma} H_2$, 
where $\Sigma = \partial H_1 = \partial H_2$. Let  $\gamma \subset \Sigma$ be a simple closed 
curve such that 
the subsurface $\Sigma \smallsetminus \nbhd(\gamma)$ is incompressible 
in $M_\infty$. 
Suppose that $M_\infty = M \smallsetminus \nbhd(\gamma)$ is a
hyperbolic manifold and $\rho:\pi_1(M_\infty) \to PSL_2(\C)$ is
a corresponding holonomy representation.
Then the group generated by the element $\rho([\gamma])$ is up to
conjugacy the unique maximal parabolic subgroup in
$\rho(\pi_1(H_1))$. 
\end{lemma}

Perhaps surprisingly, we are {\em not} assuming that $\Sigma\ssm \NN(\gamma)$ is
acylindrical in either $H_1$ or $H_2$. Instead, the compressibility of
the handlebodies plays an important role. 

\begin{proof} 
The relative core theorem (Theorem \ref{relative core}), applied to the manifold $N=\Hyp^3/\rho(\pi_1(H_1))$, gives 
a compact  core $C_1\subset N$ such that the parabolic conjugacy classes of $\rho(\pi_1(H_1))$ are represented 
by a system  of disjoint closed curves on $\boundary C_1$. The lift to $N$ of $H_1$ is also a compact core, hence 
it is  homeomorphic to $C_1$, by a map which induces $\rho$ on $\pi_1(H_1)$ (Theorems 1 and 2 of 
McCullough-Miller-Swarup \cite{mccullough-miller-swarup}). It follows that  the generators of parabolic 
subgroups of $\rho|_{\pi_1(H_1)}$ are also represented by a disjoint collection of simple closed loops 
$\beta_1,\ldots,\beta_k$ on $\Sigma = \boundary H_1$, where $\gamma =\beta_1$. 

Choose some  $\beta=\beta_i$.  Since the only parabolic conjugacy  classes in
$\pi_1(M_\infty)$ are the elements in  $\pi_1(\partial M_\infty)$, there is a singular annulus   
$\alpha : S^1 \times [0, 1] \to M_\infty$  which  on $a_0 \equiv S^1 \times \{0\}$ is a parametrization 
of $\beta$,  and maps $a_1 \equiv S^1 \times \{1\}$ to $\boundary M_\infty$.

Perturb $\alpha$, if need be, so that $\alpha \pitchfork \Sigma\ssm\NN(\gamma)$, and choose it to minimize 
the number of components of $\alpha^{-1}(\Sigma)$. Now using the fact that $\Sigma\ssm\NN(\gamma)$ 
is incompressible in $M_\infty$, we may assume that all components of $\alpha^{-1}(\Sigma)$ are  essential 
simple closed curves in $S^1 \times [0,1 ]$, or arcs in $S^1 \times [0, 1]$ with both end points on $a_1$. 
We will next show that such arcs do not occur.

Let  $\delta \subset S^1 \times [0,1 ]$ be an outermost arc of $\alpha^{-1}(\Sigma)$. The points 
$\partial \delta$ bound an arc $\delta'\subset a_1$ such that $\delta \cup\delta'$ bounds a disk 
$\Delta \subset S^1 \times [0,1 ]$. The image $\alpha(\Delta)$ is  contained in exactly one of 
either $H_1$  or $H_2$, call it $H_k$. 

Let $B$ be the annulus $\boundary M_\infty\intersect H_k$. The arc
$\delta'$ is properly embedded in $B$. If its endpoints are on the
same component of $\boundary B$, we could homotope $\delta'$  rel
endpoints to $\boundary B$. This allows us deform $\alpha$ in a
neighborhood of $\delta'$, so that $\alpha^{-1}(\Sigma)$ in a
neighborhood of $\Delta$ becomes a closed loop. 
This loop can be removed, again because $\Sigma\ssm\NN(\gamma)$ is
incompressible, contradicting the choice of $\alpha$. 

If the endpoints of $\delta'$ are on different components of
$\boundary B$, let $\gamma'$ be a core of $B$. Then the singular disk
$\alpha(\Delta)$, after a homotopy near $\delta'$,  intersects
$\gamma'$ in a single point. In other words, $\gamma'$ is primitive in
$ \widehat H_k = H_k \ssm \NN(\gamma)$, i.e.  it is isotopic to the core of a 1-handle.
That would imply  that the surface  $\Sigma \ssm \NN(\gamma)$ is compressible
in $H_k$, in contradiction to assumption $(2)$. We conclude that
$\alpha^{-1}(\Sigma)$ contains no arc components. 

The essential simple closed curves in $\alpha^{-1}(\Sigma)$
partition $S^1 \times [0,1 ]$ into sub-annuli, each of which maps into
either $H_1$ or $H_2$. Order the annuli as $A_0,A_1,\ldots,A_m$ where
$A_0$ is adjacent to $a_0$. Now, $\alpha|_{A_0}$ is a singular annulus in a
pared manifold $(H_k,\boundary H_k\intersect \NN(\gamma))$ whose
boundary is incompressible. Applying the characteristic submanifold
theory (see Bonahon \cite{bonahon:handbook} for a survey, and for
proofs see  Jaco \cite{jaco:lectures}, Jaco-Shalen \cite{jaco-shalen} and 
Johannson \cite{johannson:JSJ}) $\alpha|_{A_0}$ is homotopic through proper
annuli to a vertical annulus in one of the fibered or $I$-bundle
pieces, and in particular since one boundary $\alpha|_{a_0}$ is
already embedded, after this homotopy we can arrange for $A_0$ to
embed. Now we can proceed by induction until the whole annulus
$\alpha$ is an embedding. In particular the last subannulus $A_m$ has
one boundary parallel to $\gamma$, and the other boundary disjoint
from $\gamma$. If $A_m$ is not boundary parallel in the handlebody
$H_l$ that contains it, then it is boundary compressible (here we are
using the fact that $H_l$ is a handlebody). After
boundary compressing we obtain an essential disk in $H_l$ disjoint
from $\gamma$, contradicting the incompressibility of
$\Sigma\ssm\NN(\gamma)$. 

Hence the last annulus is boundary parallel and can be removed. Proceeding by induction 
we conclude that $\beta$ is isotopic to $\gamma$ in $\boundary H_1$. 

\end{proof}

\begin{proof}[Proof of Proposition \ref{mainproposition}]
Consider the manifold $N = \mathbb{H}^3/ \rho(\pi_1(H_1))$, which is the cover of 
$M_\infty$ corresponding  to  $\rho(\pi_1(H_1))$.

After an isotopy in $M_\infty$ we can arrange that $\gamma$ lies on
the boundary of the cusp tube, and $H_1$ meets the tube only in the annulus
$C(\gamma)=\NN(\gamma)\intersect \boundary H_1$. Lifting this embedding to $N$
yields an embedding $i:H_1\to N^0= N \ssm Q$, which takes $C(\gamma)$ to the boundary of a
cusp tube $Q$. By Lemma \ref{uniqueelement}, $Q$ is the unique cusp of $N$, and it 
follows from this   that $i(H_1)$ is a relative compact core for $N^0$ (see Section \ref{background}).  
In particular,  the ends of  $N^0$ are in one-to-one correspondence with the components 
of  $\partial H_1 \smallsetminus \gamma$.

Moreover, the Tameness Theorem of Bonahon \cite{bonahon} tells us that each end of $N^0$ is 
either {\em geometrically finite} or {\em simply degenerate}. We wish to rule out the latter.

Let $W$ be a component of $\partial H_1 \smallsetminus \nbhd(\gamma)$. If the end $E_W$ of  
$N\smallsetminus Q$ associated to $W$ is simply degenerate, then  Thurston's covering theorem 
(see \cite{canary:covering} and \cite{wpt:notes}) implies that the covering map $\varphi:N\to M_\infty$ is 
virtually an infinite-cyclic  cover of a manifold that fibers over the circle. That is, there are finite covers 
$p:\widehat N\to N$ and  $q:\widehat M_\infty \to M_\infty$, such that $\widehat M_\infty$ fibers 
over the circle and in the commutative diagram 

$$ \xymatrix{\widehat N \ar[d]^p\ar[r]^\zeta& \widehat M_\infty \ar[d]^q\\ N  \ar[r]^\varphi & M_\infty} $$ 

\vskip5pt

\noindent the map $\zeta$ is the associated infinite cyclic cover. In
particular, $\widehat N$  minus its parabolic  cusps is a product $\widehat W
\times \R$ where $\partial\widehat W \times \R$ is the preimage of the cusp boundary  of $N$.

The core $i(H_1)$ of $N$ lifts to a core $\widehat H_1\subset \widehat N$, which must be of the form 
$\widehat W \times [0,1]$ up to isotopy. 

\begin{lemma}\label{I bundle covers}
If $F$ is an orientable compact surface (not a sphere or a disk) and 
$$p:(F\times[0,1],\boundary F\times[0,1])\to (B,P) $$
is a covering of pairs. Then $B$ is an $I$-bundle, and the
cover is standard, i.e. respects the $I$-bundle structures. 
\end{lemma}

\begin{proof}
Suppose that $D_0 = p(F\times\{0\})$ and $D_1 = p(F\times\{1\})$ are distinct components 
of $\boundary B \ssm P$. Then any loop $\alpha$ in $B$ based at $D_0$ lifts to an arc with both 
endpoints on $F\times\{0\}$. This arc retracts to $F\times\{0\}$, so downstairs $\alpha$ retracts 
to $D_0$. We  conclude that $D_0 \to B$ is surjective on $\pi_1$. It is also
injective, since it is the inclusion of an incompressible surface
followed by a covering map. 
By a theorem of Waldhausen (see for example Theorem 10.2 of \cite{hempel})
$B$ is a product $G\times[0,1]$. Now it is clear that $p$ is standard,  
in particular it induced by a cover $F\to G$. 

If $D_0=D_1$, consider the subgroup $H < \pi_1(B)$ corresponding to loops based at $x\in D_0$ 
that lift to arcs whose endpoints are on the same boundary component of $F\times[0,1]$. This index 
$2$ subgroup corresponds to a degree $2$ cover $r:(\widehat B,\widehat P) \to (B,P)$, in
which $D_0$ lifts to two homeomorphic copies since $\pi_1(D_0)< H$. Now $p$ factors through a 
cover $\widehat p:F\times[0,1]\to \widehat B$ which by the previous paragraph is standard. In particular 
$\widehat B$ is a product. It now follows from Theorem 10.3 of Hempel \cite{hempel} that  $r$ is exactly 
the standard  covering from a product $I$-bundle to a twisted $I$-bundle. 

\end{proof}

Applying this lemma to the covering $\widehat H_1 \to i(H_1)$, we see that $(H_1,C(\gamma))$ is a 
(possibly twisted) $I$-bundle  where $C(\gamma)$ is the sub-bundle over the boundary of the base
surface. We claim now that the complement $M_\infty\smallsetminus H_1$ is also an $I$-bundle

Since $\widehat M_\infty$ fibers over the circle and $\widehat W$ maps to a fibre,  $\widehat H_1$ 
embeds in  $\widehat M_\infty$, and its complement is also a product $I$-bundle, $X\times I$. Consider
any component $\widetilde{H_1}\ne \widehat H_1$ of the preimage $q^{-1}(H_1)$. Since it covers $H_1$, 
it too is an $I$-bundle. Its base surface $Y$ is an embedded incompressible surface in $X\times I$, with
$\boundary Y$ in the vertical annuli $\boundary X\times I$  (since $q$ is a covering of pared manifolds). 
By Lemma 5.3 of Waldhausen \cite{waldhausen}, this implies that $Y$ is isotopic to $X \times\{1/2\}$, and 
hence $\widetilde{H_1}$ is isotopic to  $X \times I'$ where $I'\subset int(I)$.  This means that  $q^{-1}(H_1)$ 
is a union of standardly embedded untwisted $I$-bundles in  $\widehat M_\infty$, and hence so is its
complement. We conclude (again by Lemma \ref{I bundle covers}) that the complement of $H_1$ in $M_\infty$ 
is also an $I$-bundle, properly embedded as a pared manifolds.

This contradicts condition (3) of the theorem,  and so concludes the proof that no end
of $N\smallsetminus Q$ is simply degenerate. 

Hence $\rho$ is geometrically finite with one cusp,  which satisfies the cut-point-free condition 
(by hypothesis (1)). Theorem \ref{blocking PS} implies that $\rho$ is primitive stable.  

\end{proof}

\section{Links carried by a Trellis}\label{KCT}

Links carried by a trellis were first defined in ~\cite{lustig-moriah:highgenus}.
We reproduce the definition here for   the convenience of the reader,
but we also make some changes in notation.

A {\em trellis} is a connected graph $\mathcal{T}$ in a vertical coordinate plane
$P  \subset \mathbb{R}^3$ which consists of horizontal and  vertical
edges only, and whose vertices have valence 2 or 3 and are of the type
pictured in Figure \ref{vertextypes}.  

\def\epsfbox#1{{BLANK BLANK}}

\begin{figure}[ht]
\centering
\includegraphics[width=3.4in]{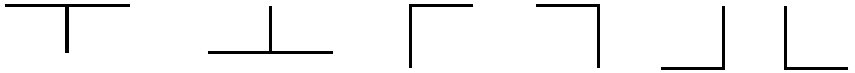}
\caption{Allowable vertex types in a trellis}
\label{vertextypes}
\end{figure}

Given a labeling of the vertical edges by integers, we can describe a
knot or link  on the boundary of a regular neighborhood of
$\mathcal T$, by giving a standard picture for the neighborhood of
each vertex and edge. This is done in Figure \ref{plumbing}. Note that
one of each combinatorial type of vertex is pictured, the rest being
obtained by reflection in the coordinate planes orthogonal to $P$. The
integer label for a vertical edge counts the number of (oriented)
half-twists. The pieces fit together in the obvious way. 
In the discussion to follow we will consistently use
{\em right/left} and {\em top/bottom} for the horizontal and vertical directions
in $P$, which is parallel to the page, and {\em front/back} for the
directions transverse to $P$  and closer/farther from the reader, respectively.
In particular $P$ cuts the regular neighborhood of $\TT$ into a a
front and a back part. 

If $a$ is the function assigning to each vertical edge $e$ its label
$a(e)$, we denote by $K[a]$ the knot or link obtained as above. We say
that $K[a]$ is {\em carried by} $\mathcal T$. 

Since $\TT$  is planar and connected, its regular neighborhood in $\mathbb{R}^3$ is a handlebody 
$H_1 = \overline{\nbhd(\mathcal{T})}$  embedded in the standard way in $S^3$, which we identify with 
the one-point compactification of $\mathbb{R}^3$. The complement  $H_2 = \overline{S^3 - \nbhd(\mathcal{T})}$  
of $int(H_1)$ is also a handlebody. The pair  $(H_1,H_2)$ is a Heegaard splitting of $S^3$, which we call a  
{\em Heegaard splitting of the pair  $(S^3, K[a])$}, or a {\it trellis Heegaard splitting}. We refer to $H_1$ as 
the {\it  inner} handlebody  and to $H_2$ as the {\it outer} handlebody of this splitting. We denote the surface 
which is their common boundary    by $\Sigma$.   Let $g(\TT)$ denote the genus of $H_1$ and $H_2$.

\vskip10pt

\begin{figure}[ht]
\centering
\includegraphics[width=3.4in]{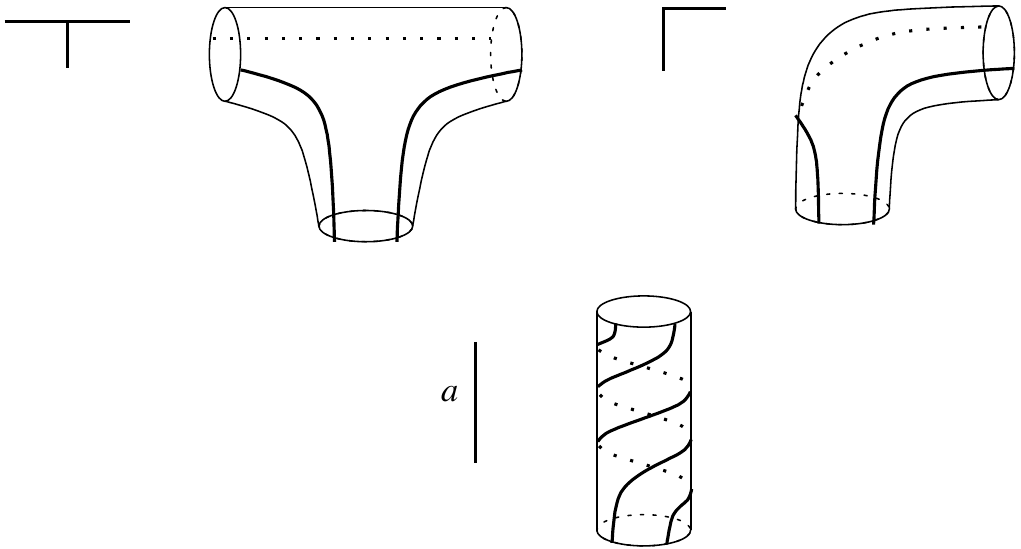}
\caption{Local types for the link carried by a trellis.
The vertical tube carries
  $a\in{\mathbb Z}$ half-twists ($a=3$ is pictured). }
\label{plumbing}
\end{figure}

\subsection {Nice flypeable trellises}

Every maximal connected union of horizontal edges of $\mathcal{T}$ is
called a {\it horizontal line}.  A trapezoidal region bounded by two
horizontal lines and containing only vertical edges in its interior is
called a {\it horizontal layer}. 

A trellis is {\it brick like of type $(b, c)$} if it is a union of $b$ layers each containing $c$ squares
arranged in such a way so that:

\begin{enumerate}

\item Vertical edges incident to a horizontal line (except the top and bottom lines) point alternately up and down. 

\item Layers are alternately ``left protruding'' and ``right protruding", where by left protruding we mean that the leftmost
vertical edge is to the left of the leftmost vertical edges in the layers both above and below it. The definition for right 
protruding and for the top and bottom layers is done in the obvious way.

\end{enumerate} 

\begin{figure}[ht]
\centering
\includegraphics[width=2.5in]{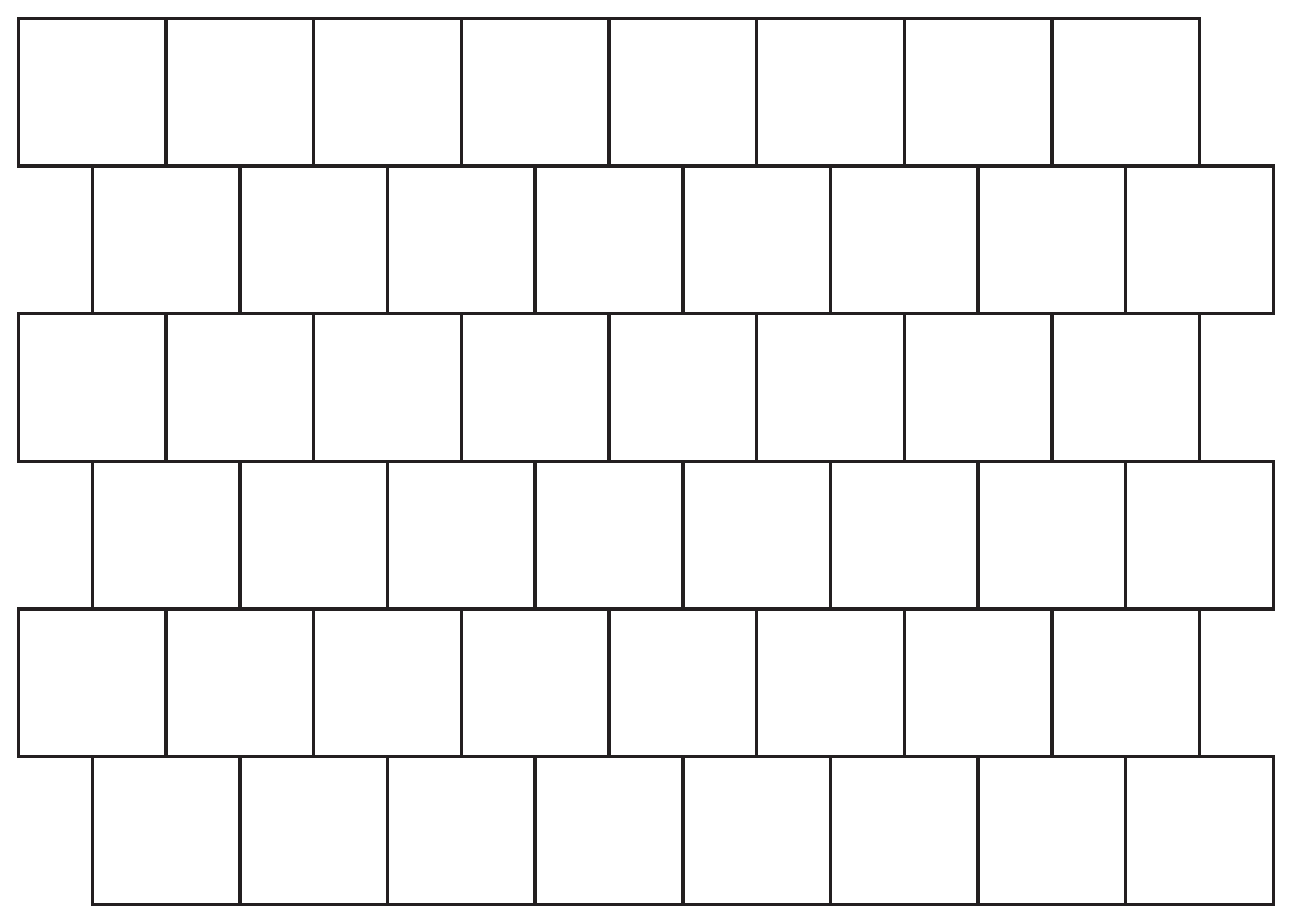}
\caption{A brick-like trellis}
\label{bricklike}
\end{figure}

A trellis is {\it flypeable} (see Figure \ref{flypeable}) if it is obtained from a brick like trellis in the following way:
Choose $ 1 < i < b$, and in the $i-th$ layer choose a contiguous sequence of squares $D_1, \dots D_r$ 
not including the leftmost or rightmost square. Now remove all vertical edges incident to the squares from the 
layers above and below. See Figure \ref{flypeable} for an example. 

A trellis is {\it nice flypeable} if $ 2 < i < b - 1$ and the squares $D_1, \dots D_r$ do 
not include the two leftmost or the two rightmost squares.

\begin{figure}[ht]
\centering
\includegraphics[width=2.5in]{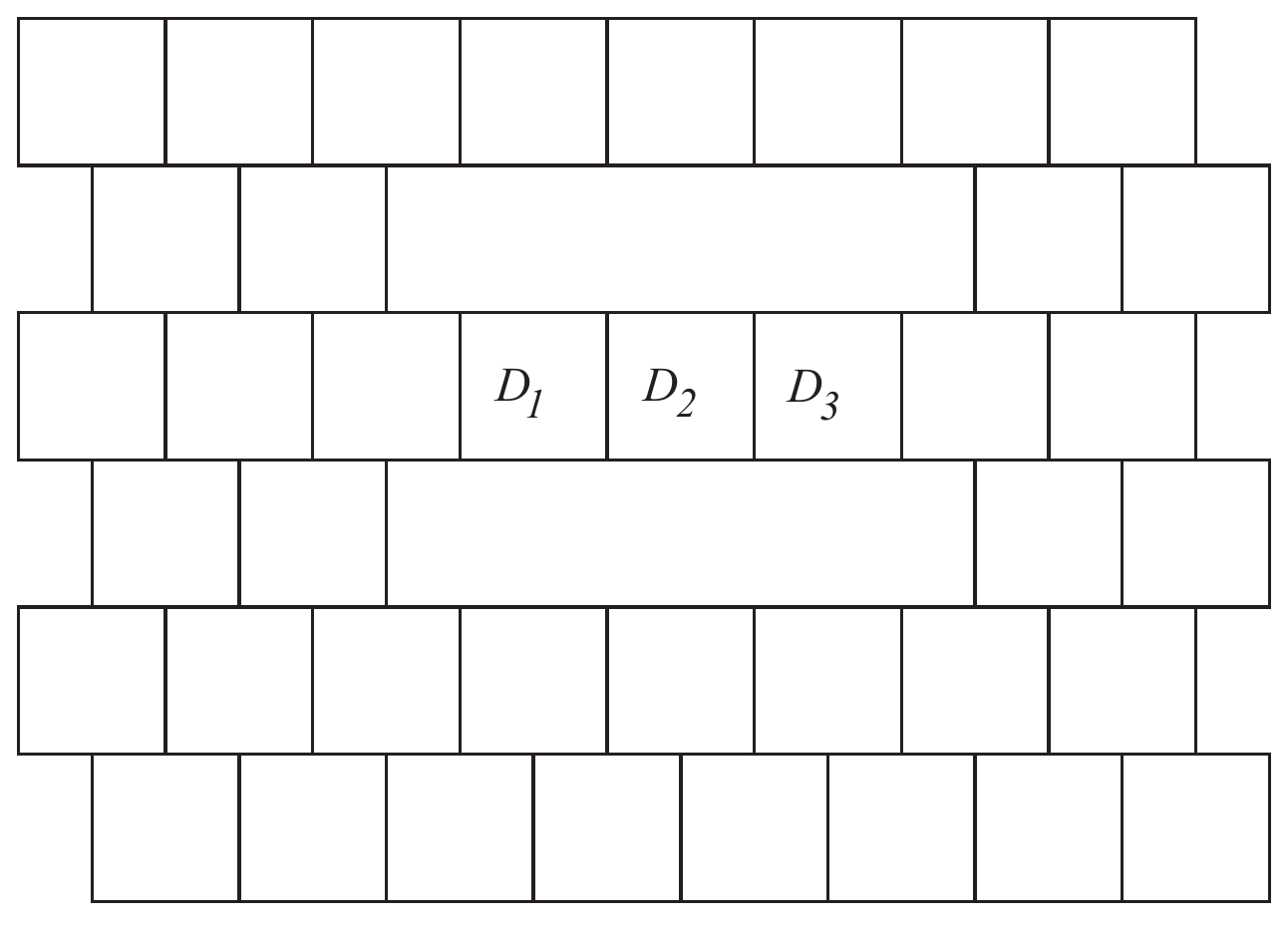}
\caption{A nice flypeable trellis}
\label{flypeable}
\end{figure}

Given a flypeable trellis carrying a knot or link $K$,  a {\it flype} is an ambient isotopy of $K$ which is obtained as follows: 
Let $R $ be the union of the squares $D_1, \dots D_r$ including their interiors. We call $R$ the {\it flype rectangle}. 
Let $B$ be a regular neighborhood of $R$.  We choose $B$ so that it contains all  subarcs of $K$ winding around 
the edges of $R$ except for  the horizontal arcs of $K$ that travel in the back of $\Sigma$ along the horizontal 
edges of $R$. Hence $\boundary B$ intersects $K$ in four points (see Figure \ref{flypebox}.) 

\begin{figure}[ht]
\centering
\includegraphics[width=3in]{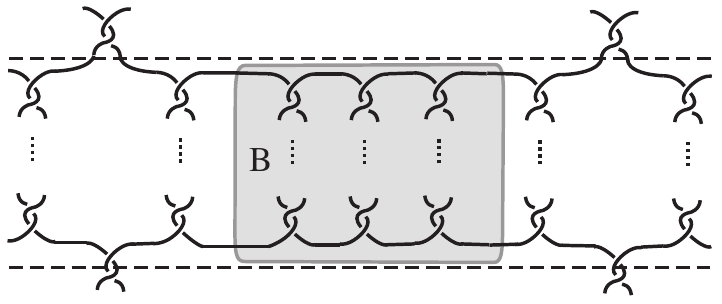}
\caption{The box $B$ contains a portion of the knot associated to adjacent vertical
  edges, but excludes the horizontal arcs passing in the back (dotted).}
\label{flypebox}
\end{figure}

\vskip10pt

A flype will flip the box $B$ by 180 degrees about a horizontal axis leaving all parts of the knot outside a small 
neighborhood of $B$ fixed.  This operation changes the projection of $K$ in $P$ by adding a crossing on the left and 
a crossing on the right side of the box. These crossings have opposite signs.

\begin{figure}[ht]
\centering
\includegraphics[width=4.5in]{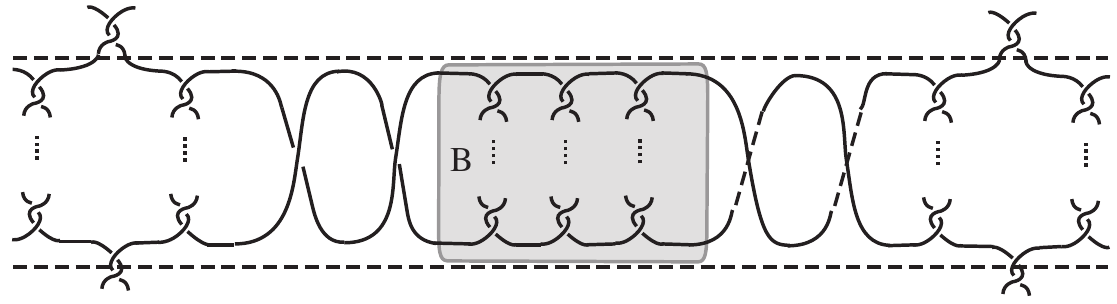}
\caption{Local picture of $K^{2}$, a double-flype move on $K$. Note  that the  genus of the new trellis 
which carries $K^{2}$ is bigger by 4 than the genus of the trellis which carried $K$.}
 \label{afterflype}
\end{figure}

The projection of $K$ obtained after a flype is carried by a new trellis. It differs
from $\mathcal{T}$   in that there is a new vertical edge on the left side of $R$ and
another new one on the right side of $R$, one with twist coefficient $1$
and the other one with $-1$.  The  flype will  be called  {\it positive} if the
coefficient of the left  new edge is positive. A  positive / negative flype
iterated $|\, r\, |$ times  will be called a $r${\it -flype}, $r \in \mathbb{Z}$, (see Figure
\ref{afterflype}). Denote the image of $K$ after the  $r$-flype by $K^{r}$ and the new
trellis with the new $2|r|$ vertical edges by $\TT^{r}$. Similarly we will denote $\nbhd(\TT^{r})$ by
$H_1^{r}$ and  $\partial H_1^{r}$ by $\Sigma^{r}$. Notice that $g(\mathcal{T}^{r}) = g(\mathcal{T}) + 2|r|$. 

\vskip10pt

The following restatement of Theorem 3.4
of~\cite{lustig-moriah:highgenus} describes the embedding of $K^r$ in
$\Sigma^r$ under suitable assumptions:

\vskip10pt

\begin{theorem}\label{main-incomp}
Let $\mathcal{T}$ be a flypeable trellis and let $K = K[a] \subset S^3$ be a  knot or link carried  by 
$\mathcal{T}$ with twist coefficients given by $a$.  Assume that $a(f) \geq 3$ for all vertical edges 
$f$ and that for the two vertical edges $e, e'$ immediately to the left and right of the flype rectangle we have
 $a(e), a(e') \geq 4$. Then for all $r \in\mathbb{Z }$, the surface $\Sigma^{r} \ssm K^{r}[a]$ is incompressible in both the
interior and the exterior handlebodies $H_1^{r}, H_2^{r}$.

\end{theorem}

  \vskip10pt

\subsection{Horizontal surgery on knots carried by a trellis}\label{HSKT}

In our construction we will need $K$ to be a knot. The number of components of $K$ is determined by the residues $a(e) \mod 2$, and it is easy to see that if $K$ has more than one component then the number can be reduced  by changing $a(e) \mod 2$ at a column where two components meet. Hence a given trellis always carries knots with arbitrarily high coefficients. We will assume from now on that $K$ is a knot.

The embedding of $K$ in $\Sigma$ defines a framing, in which the longitude $\lambda_\Sigma$ is a boundary component of a regular neighborhood of $K$ in $\Sigma$. We let $K_\Sigma(p/q)$ denote the result of $p/q$ surgery on $S^3\ssm K$ with respect to this framing. 

In particular $K_\Sigma(1/m)$, for $m\in\Z$, will be called  a {\em horizontal } Dehn surgery on $K$ with respect to $\Sigma$. Note that  it has the same effect as cutting $S^3$ open along $\Sigma$ and regluing by the $m^{th}$ power of a Dehn twist on $K$.

It is interesting to note that a flype does not change this framing, i.e.
$$K_{\Sigma}(p/q) = K^r_{\Sigma^r}(p/q)$$
for all $p/q$ (see \cite{lustig-moriah:highgenus}). This is because the effects on the framing from the new 
crossings on both sides of the flypebox cancel each other out. We will
not, however, need this fact in 
our construction.

\section{Satisfying conditions for primitive stability}\label{SCPS}

\vskip10pt

In this section we will consider representations for manifolds obtained from diagrams of flyped knots on
a nice flypeable trellis. We show that they are hyperbolic and that they satisfy the hypotheses required 
 by Proposition \ref{mainproposition}.

\subsection{Whitehead graph}

Fix  $r \in\mathbb{Z}$. Let $\{e_1,\dots , e_{n_r}\}$ denote the set of vertical edges of $\TT^r$ not including the 
rightmost one in each layer. Each $e_i$ is dual to a disk $\Delta_i$ in $H_1^r$, and note that these disks cut 
$H^r_1$ into a 3-ball, hence $n_r = g(H^r_1)$. Let $x_i$ be the generator of $\pi_1(H^r_1) = F_{n_r}$ which is 
dual to $\Delta_i$ and let $X = \{x_1,\ldots,x_{n_r}\}$. The curve $K^{r}$ contains no arc that meets a disk 
$\Delta_i$ from the same side at each endpoint without meeting other $\Delta_j$ in its interior, and it follows 
that $K^r$  determines a cyclically reduced word $[K^{r}]$  in the generators   $X$.

\vskip10pt

\begin{lemma} \label{WHG}The Whitehead graph  $Wh([K^{r}], X)$ is cut point free for each 
$r \in\mathbb{Z}$.

\end{lemma}

\vskip10pt

\begin{proof} 
 A regular neighborhood of each $\Delta_i$ in $H^r_1$ is bounded by two disks $\Delta_i^\pm$. Let $Q$ denote 
 $H^r_1$ minus these regular neighborhoods. Then $K^r\cap Q$ is a collection of arcs corresponding to the edges 
 of the Whitehead graph, and the disks $\Delta_i^\pm$ represent the vertices. After collapsing each disk to a point 
 we get the Whitehead graph itself. Since $K^r$ meets each $\Delta_i$ in exactly two points, this graph  is necessarily 
 a 1-manifold. Now we observe that one can isotope this 1-manifold in a neighborhood of each horizontal line of the 
 trellis so that the vertical coordinate of the plane $P$ is a height function on it with exactly one minimum and one 
 maximum. Hence it is a circle, and in particular cut point free.

\end{proof}

\begin{figure}[ht]
\centering
\includegraphics[width=3.5in]{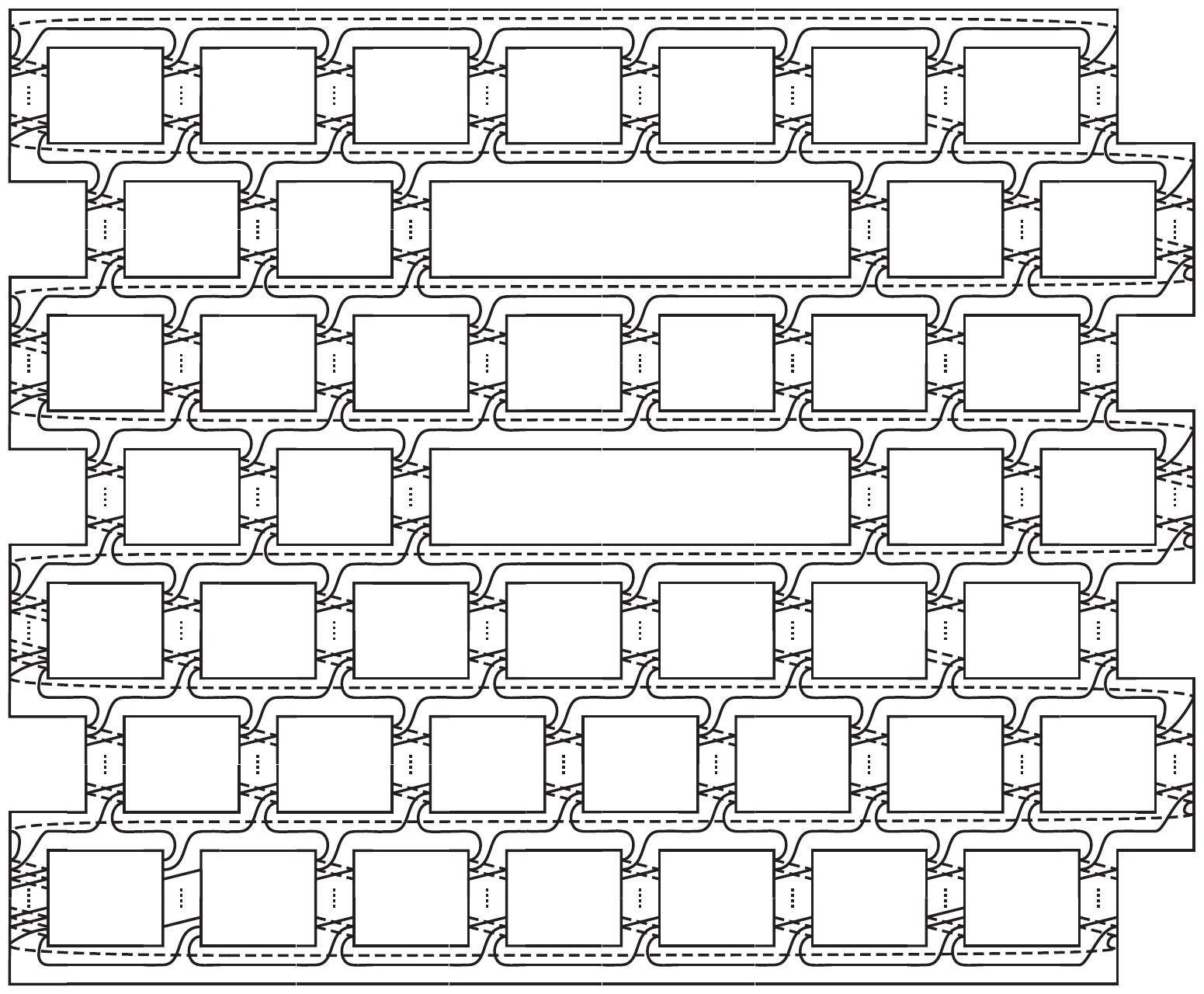}
\caption{A representative nice flypeable trellis carrying a knot.}
\label{KnotcarriedbyNFT}
\end{figure}

\subsection{Hyperbolicity}
\vskip5pt

Next we prove the hyperbolicity of our knot complements: 
\begin{theorem}\label{hyperbolic} 
Let $\mathcal{T}$ be a  nice flypeable trellis and let $K = K[a] \subset S^3$ be a  knot 
or link carried  by  $\mathcal{T}$. Assume that  $ a(f)\geq 3$ 
for all vertical edges $f$. Assume further that the  pair of edges $(e,e')$ at the sides of the 
flype region have twist coefficients $a(e), a(e') \geq 4$.
Then  $X = S^3 \smallsetminus \nbhd(K)$  is a hyperbolic manifold.
\end{theorem}

Note that this gives us hyperbolicity of $S^3 \ssm\NN(K^r)$ for
all $r\in \Z$, since $\{K^r\}$ are all isotopic.

\begin{proof} 
Recall that $H_1$ and $H_2$ are the interior and exterior
handlebodies, respectively, of the trellis on which $K$ is defined. 
We first reduce the theorem to statements about annuli in $H_2$:

\begin{lemma}\label{suffices to prove}  If  the pared manifold 
$(H_2 ,\bar\NN_{\Sigma}(K))$ is acylindrical,   then $X$ has  no essential tori. 
\end{lemma} 

\begin{proof}
Let  $T \subset X$ be an incompressible torus, and let us prove that it
is boundary-parallel. 
Choose $T$ to intersect the surface $\Sigma\ssm \nbhd(K)$
transversally and with a minimal number of components.

The intersection $T \cap  (\Sigma \ssm \nbhd(K))$ must  be non-empty
since handlebodies do not contain incompressible tori.  Since 
$K$ satisfies the conditions of Proposition \ref{main-incomp}, the
surface $\Sigma \ssm \nbhd(K)$ is incompressible in both the  inner
and outer handlebodies $H_1$ and $H_2$.  

The intersection $T  \cap \Sigma$ cannot contain essential curves in
$\Sigma \ssm \nbhd(K)$  which are inessential in $T$ and essential
curves in $T$ which are inessential in  $\Sigma\ssm  \nbhd(K)$ as this
would violate the fact that both surfaces are incompressible. Curves
that are inessential in both surfaces are ruled out by minimality.
Hence $T \cap H_i$ ($i = 1,2$) is a collection of essential annuli in
$T$ which are incompressible in $H_i \ssm \NN(K)$. By 
minimality they are not  parallel to $\boundary H_i \ssm \NN(K)$.  

By the hypothesis of the lemma, 
$T\intersect H_2$ is a union of concentric
annuli parallel to a neighborhood of $K$ in $\boundary H_2$.

Suppose first that there is a single such annulus $A$. If we push $A$
to a component $A_1$ of a neighborhood of $K$ in $\boundary H_2$,
we obtain a torus $T'$ in $H_1$ which is homotopically
nontrivial, and hence bounds a solid torus $V$ in $H_1$. The 
intersection $V\intersect \boundary H_1$ is $A_1$. If $A_1$ is a
primitive annulus in $V$ then so is its complement, which is just
$T'\intersect H_1 = T\intersect H_1$. Hence $T\intersect H_1$ is boundary parallel to
$A_1$, so $T$ is an inessential torus in $M_\infty$, and we are done. 

If $A_1$ is not primitive in $V$, then we have exhibited the
handlebody $H_1$ as a union $V\union W$, where $W$ is a handlebody of
genus greater than one and $W\intersect V$ is an annulus $A_2$ in $\boundary W$. 
Now $A_2$ cannot be primitive in $W$, 
since $\boundary W \ssm A_2 = \Sigma\ssm A_1$, which is
incompressible. But the gluing of two handlebodies along an annulus
produces a handlebody only if the annulus is primitive to at
least one side. (This follows from the fact that any annulus in a
handlebody has a boundary compression, which determines a disk on one
side intersecting the core of the annulus in a single point.)
This is a contradiction. 

Now if $T\intersect H_2$ consists of more than one annulus, then one
of the annuli, together with a neighborhood of $K$, bounds a solid
torus $U$ in $H_2$ which contains all of the other annuli. Naming the
annuli $B_1,\ldots B_n$, where $B_1$ is the outermost, let $U'$ denote
$U$ minus a regular neighborhood of $B_1$, so that $U'$ 
contains $B_2\union \cdots \union B_n$, and
let $H'_1 = H_1 \union U'$. We can isotope $K$ through $U'$ to a knot
$K'$ on $\boundary H'_1$, so that the isotopy intersects $T$ in a
disjoint union of cores of $B_2,\ldots,B_n$. Now $H'_2 = cl(M\ssm
H'_1)$ intersects $T$ in the single annulus $B_2$, and we can apply
the previous argument to show that $T$ bounds a solid torus $V$ in which
$K'$ is primitive. If $n$ is even, then $K$ is outside $V$, so that
$T$ is inessential already. If $n$ is odd then $K$ is inside $V$, 
and since the isotopy from $K$ to $K'$ passed through a
sequence of disjoint curves in $T$, we conclude that $K$ is itself
primitive in $V$. Hence again $T$ is inessential. 

\end{proof}

\vskip4pt

\begin{proposition}\label{boundaryparallelannuli} Let $K$ be a knot carried by a nice 
flypeable trellis so that $a(e) \geq 3$ for every vertical edge $e$. Then there 
are no essential annuli in $(H_2, \NN_{\Sigma}(K))$.

\end{proposition}

\vskip5pt

Before we  prove this proposition  we need some definitions and
notation. The proof will be somewhat technical and enumerative, but
the notation and data we will set up will then be useful in proving
Proposition \ref{flypedtrellis} which applies to the flyped case. 

\vskip4pt

Denote by $\DD$ the collection of disks in $P$ which are the bounded
regions of $P \ssm  \NN_P(\TT)$.
The front side of each $D \in \DD$ will be denoted by $D^+$ and the back side by $D^-$. Set 
$\widehat{\DD} =\{ D^+, D^- \, : \,D \in \DD\}$, the set of ``disk
sides''. The number of vertical edges of $\TT$ adjacent to the top
(resp. bottom)  
edge of a disk $D$ is denoted by $tv(D)$ (resp. $bv(D)$).  
Each $D\in\DD$ is contained in a single component of $P\ssm\TT$, 
and we sometimes abuse notation by calling this larger disk $D$ as
well. 

\vskip10pt

\begin{proof}[Proof of Proposition \ref{boundaryparallelannuli}] Assume  that $A$ is a properly embedded incompressible
 annulus in $(H_2,\bar\NN_{\Sigma}(K))$ which is not parallel to $\partial H_2 \ssm \NN(K)$. We will show that it is parallel to a 
 neighborhood of $K$ in $\partial H_2$ and hence not essential.

The proof will be in two stages. In the first step we will show that,
after isotopy,  any such annulus $A$  can be decomposed as a
cycle of rectangles where two adjacent rectangles meet along an arc of
intersection of $A$ with $\DD$. 
In the second step we show that a cycle of rectangles must be parallel to the knot.
Throughout we will abuse notation by referring to $\union_{D\in \DD}
D$ as just  $\DD$.

\vskip5pt

\noindent \underline{Step 1:} The disks of $\DD$ are essential disks
in the outer handlebody $H_2$, and
$H_2 \smallsetminus \nbhd(\mathcal{D})$  is a $3$-ball. Note that the
disk-sides in $\widehat\DD$ can be identified with disks of
$\boundary \nbhd(\DD)$ that lie in the boundary of this ball.  Isotope
$A$ to intersect $\DD$ transversally and with a minimal number of components.  The
intersection must be nonempty, as otherwise  
$A$ will be contained in a $3$-ball and will not be essential.  No
component of $A \cap \mathcal{D}$  
can be a  simple closed curve,  since this would either violate the 
the fact that  $A$ is  essential, or allow us to reduce the number of components 
in  $A \cap \mathcal{D}$ by cutting and pasting.
 
Let $\EE$ denote the set of components of $\Sigma\ssm (K\union \DD)$. 
By Lemma 2.2 of \cite{lustig-moriah:highgenus} we have that each $E\in
\EE$ is a 2-cell, and that the intersection of $\boundary E$ with
any disk $D \in \mathcal{D}$ is either

\begin{enumerate}
\item empty or 

\vskip4pt

\item consists of precisely one arc or 

\vskip4pt

\item consists of precisely two arcs along which $E$ meets $\DD$  from opposite sides of $P$.
\end{enumerate}
(One can also obtain these facts from the enumeration that we will
shortly describe of all local configurations of $\EE$, and in fact we
will generalize this later). 

We claim the arcs of intersection in  $A \intersect \DD$ must  be essential in
$A$: 
If not, consider an outermost 
such arc  $\delta$ which, together with a subarc $\alpha \subset \partial A$, bounds a sub-disk 
$\Delta_1$ in $A$. Let $D$ be the component of $\DD$ containing
$\delta$. By transversality, a neighborhood  
of $\delta$ in $\Delta_1$ exits $P$  from only one side, and hence both ends of $\alpha$ meet $\delta$ 
from the same side of $P$. The arc  $\alpha$  is contained in a single component $E$ of 
$\Sigma \ssm (K\union \DD)$. If $\boundary E$ meets $D$ in more than
one arc then, by (3) above, it does so from opposite  
sides of $P$. Hence the endpoints of   $\alpha$ can meet only one such
arc in $\partial E$. Since $E$ is  
a polygon,  $\alpha$ together with a subarc $\gamma$ of  
$\boundary D\intersect \boundary E$ bound a sub-disk  $\Delta_2
\subset E$.  Now $\gamma$ and $\delta$ must bound a sub-disk $\Delta_3$
in $D$.
The union  $\Delta_1 \cup \Delta_2 \cup \Delta_3$ is a $2$-sphere in the complement of $K$ which bounds a  
$3$-ball in $X$. Hence we can isotope $A$ to reduce the number of components in   
$A \cap \mathcal{D}$ in contradiction to the choice of $A$. We conclude that all arcs of $A \intersect \mathcal D$  
are  essential in $A$.

As the intersection arcs in  $A \cap \mathcal{D}$ are essential in $A$
they cut it into rectangles. We can summarize this structure in the following lemma:
\begin{lemma}\label{rectangle structure}
After proper isotopy in $(H_2,\bar\NN_\Sigma(K))$, an incompressible
annulus $A$ intersects $\DD$ in a set of essential arcs, which cut $A$
into rectangles. The boundary of a rectangle can be written as a union
of arcs $\delta_1,\delta_2\subset A\intersect\DD$, and arcs
$\alpha_1,\alpha_2 \subset \boundary A$, so that $\alpha_1$ and
$\alpha_2$ are contained in {\em distinct} components $E_1,E_2\in\EE$.
\end{lemma}

What remains to prove is just the statement that the components $E_1$
and $E_2$ of $\EE$ containing $\alpha_1$ and $\alpha_2$,
respectively, are distinct. 

Choose an arc $\delta$ of $A\intersect \DD$, and let $D$ be the disk
in $\DD$ containing $\delta$. We claim that $\boundary \delta$ separates
the points of  $K \cap \boundary D$ 
on $\partial D$. For if not, then $\delta$ would define a boundary
compression disk for $A$  in  $H_2$ which  
misses $K$.  After boundary compressing $A$ we will obtain (since $A$
is not parallel to $\boundary H_2 \ssm\NN(K)$) an essential disk $\Delta \subset H_2$ 
with $\partial \Delta \subset \Sigma \ssm \nbhd(K)$, which contradicts
Proposition  \ref  {main-incomp}. 

Now for a rectangle $R$ of $A\setminus \DD$, since each $\delta_i$ in
$\boundary R$ 
separates $K \cap \boundary D$ on $\partial D$, the arcs $\alpha_1$
and $\alpha_2$ meet $\boundary D$ in different arcs of $\boundary D
\ssm K$. Using transversality as before, $\alpha_1$ and $\alpha_2$
exit $D$ from the same side of $P$, and hence by properties (1-3)
above (from Lemma 2.2 of \cite{lustig-moriah:highgenus}),
they cannot be contained in the same component $E$ of $\EE$. We
conclude that $E_1\ne E_2$. This completes the proof of Lemma
\ref{rectangle structure}. 

\medskip

We will adopt the notation $(D^{e_1}_1, D^{e_2}_2, E_1, E_2)$ to
describe the data that determine a rectangle up to isotopy, where $D^{e_i}_i \in
\hhat \DD$ and $E_i\in\EE$. That is, the rectangle meets $D_i$ along
its boundary arcs $\delta_i$ on the sides determined by $e_i$, and the
arcs $\alpha_i$ are contained in $E_i$. We call a rectangle {\em
  trivial} if 
$E_1$ and $E_2$  are adjacent along a single sub-arc of $K \ssm
\DD$. This is because the rectangle can then be isotoped into a
regular neighborhood of this sub-arc. 

\vskip5pt

\noindent \underline{Step 2:} We now show that if $H_2 = S^3 \ssm
\NN(\TT)$, where $\mathcal{T}$ is a nice  
flypeable trellis, one cannot embed in   $\partial H_2  \smallsetminus 
\nbhd(K)$ a sequence of non-trivial   
rectangles  $R$,  as above,  which fit together to compose an 
essential annulus $A$.  

We first prove the following lemma:

\begin{lemma}\label{single layer} If $(D^{e_1}_1, D^{e_2}_2, E_1,
  E_2)$ determines a non-trivial rectangle  
then $D_1, D_2$ are contained in a single layer of the trellis.

\end{lemma}

After this we will prove Lemma \ref{enumerate rectangles} which enumerates the types of nontrivial rectangles 
which do occur, and Lemma \ref{rectangle adjacencies} which describes the ways in which rectangles can be 
adjacent along their intersections with $\DD$. We will then be able to see that the adjacency graph of nontrivial 
rectangles contains no cycles, which will complete the proof of Proposition \ref{boundaryparallelannuli}.

The proof will be achieved by a careful enumeration of how disk-sides in $\widehat{\DD}$ are connected by regions 
in $\EE$. We will examine each type of disk in $\DD$ on a case by case basis. For each disk we will consider 
only  its connection to disks along $\EE$ regions meeting it on the top and sides. The complete picture can be 
obtained using the fact that a $180$ rotation in $P$ of a nice flypeable trellis is also a nice flypeable trellis 
(see Figure \ref{KnotcarriedbyNFT}).

\begin{proof}[Proof of Lemma \ref{single layer}]

We use the following notation:

\begin{enumerate}

\item Connectivity:  The symbol $$D^\pm  \stackrel{E}{\longleftrightarrow} D^\pm_1, \dots, D^\pm_n$$

\noindent where $D^\pm, D^\pm_1, \dots, D^\pm_n \in \widehat{\DD}$ and $E \in \EE$, means that $E$ meets the disks 
$D, D_1, \dots, D_n$ on the indicated sides. (Although the asymmetry
of treating one disk differently from the others seems artificial
here, it is suited to the order in which we enumerate cases).

\item Disk coordinates:  When considering a given disk $D$ we will use relative ``cartesian'' coordinates $D_{i,j}$
for $D$ and its neighbors, where $D = D_{0,0}$,  $i$ indicates layer and $j$ enumerates disks in a layer from 
left to right.

\item $\EE$ region coordinates: The $\EE$ region adjacent to the top edge of a disk will always be enumerated by $E_0$. 
Regions along the left and right edges will be enumerated  in a clockwise direction by consecutive integers.
Note that we do not enumerate the  regions near the bottom of a disk and they can be understood by symmetry.

\end{enumerate}


First we enumerate adjacencies for ``front'' disks $D^+$. In each case we will give a precise figure for the local 
configuration and a list of connectivity data which can be verified by inspection.


\subsection*{Middle disks}

Begin with  disks which are not leftmost or rightmost in their layer. Cases will  be separated depending on
the top valency of the disk in question.

\begin{enumerate}

\item[(A1)]  $tv(D) = 1$. The neighborhood of the top edge of $D = D_{0,0}$ is described  in Figure \ref{tv 1 interior}.

\begin{figure}[ht]
\centering
\includegraphics[width=4in]{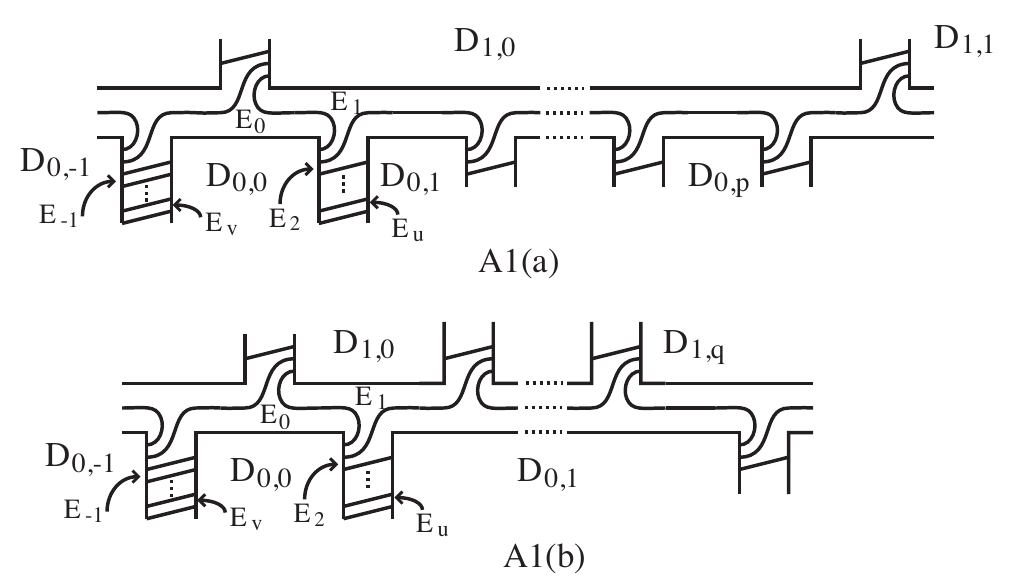}
\caption{This figure describes the case discussed in A1(a) and A1(b).}
\label{tv 1 interior}
\end{figure}

\noindent Note that there are two cases, depending on whether $tv(D_{0,1})$ is zero or not. 

\vskip7pt

\begin{enumerate}

\item $tv(D_{0,1}) = 0$: 
\begin{align*}
D^+_{0,0}  &\stackrel{E_0}{\longleftrightarrow} D_{1,0}^+, D_{0, -1}^+\\
D^+_{0,0}  &\stackrel{E_1}{\longleftrightarrow} D_{1,0}^+, D_{0,1}^+, \dots, D^+_{0,p}, D^+_{1,1}\\
D^+_{0,0}  &\stackrel{E_j}{\longleftrightarrow} D_{0, 1}^+, \,\,  2 \leq j \leq u\\ 
D^+_{0,0}  &\stackrel{E_k}{\longleftrightarrow} D_{0, -1}^+, \,\,\, v \leq k  \leq -1\\
\end{align*}


\item $tv(D_{0,1}) > 0$: 
\begin{align*}
D^+_{0,0}  &\stackrel{E_0}{\longleftrightarrow} D_{1,0}^+, D_{0, -1}^+\\
D^+_{0,0}  &\stackrel{E_1}{\longleftrightarrow} D_{1,0}^+, D_{1, 1}^+\\
D^+_{0,0}  &\stackrel{E_2}{\longleftrightarrow} D_{0,1}^+, D_{1,1}^+, \dots, D^+_{1,q}\\
D^+_{0,0}  &\stackrel{E_j}{\longleftrightarrow} D_{0, 1}^+, \,\,\, 3 \leq j \leq u\\ 
D^+_{0,0}  &\stackrel{E_k}{\longleftrightarrow} D_{0, -1}^+, \,\,\, v \leq k  \leq -1\\
\end{align*}

\end{enumerate}

\item[(A2)] $tv(D)  > 1$ (see Figure \ref{tv > 1 interior}). Note that $D_{1,s}$ is not the rightmost disk in 
its layer, by the  nice flypeable condition.

\begin{align*}
D^+_{0,0}  &\stackrel{E_0}{\longleftrightarrow} D_{1,0}^+, \dots, D^+_{1,s}, D_{0, -1}^+\\
D^+_{0,0}  &\stackrel{E_1}{\longleftrightarrow} D_{1,s}^+, D_{1, s + 1}^+\\
D^+_{0,0}  &\stackrel{E_2}{\longleftrightarrow} D_{1,s + 1}^+, D_{0,1}^+ \\ 
D^+_{0,0}  &\stackrel{E_j}{\longleftrightarrow} D_{0, 1}^+, \,\,\, 3 \leq j \leq u\\ 
D^+_{0,0}  &\stackrel{E_k}{\longleftrightarrow} D_{0, -1}^+, \,\,\, v \leq k  \leq -1\\
\end{align*}

\begin{figure}[ht]
\centering
\includegraphics[width=3.5in]{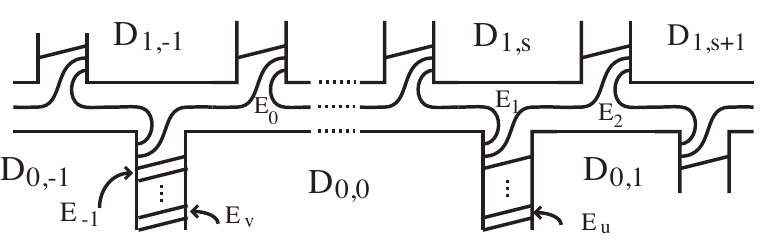}
\caption{The case  where $tv(D) > 1$, discussed in A2.}
\label{tv > 1 interior}
\end{figure}

\item[(A3)]  $tv(D)  = 0$.

\begin{enumerate}

\item $D$ is not in the top row (see Figure \ref{A3a}). In this case $D$ can be one of a sequence 
$D_{0,s}, \dots, D_{0,0}  , \dots, D_{0,t}, \,\,  (s \le 0 \le t)$ of disks whose top edge is adjacent 
to the bottom edge of $D_{1,0}$. If it is not the  rightmost one in the sequence (i.e. $ t > 0$) then:

\begin{align*}
D^+_{0,0}  &\stackrel{E_0}{\longleftrightarrow} D_{0, -1}^+\\
D^+_{0,0}  &\stackrel{E_1}{\longleftrightarrow} D_{1,0}^+, D_{1,1}^+, D^+_{0,s - 1}, D^+_{0,s},\dots, D^+_{0,t} \\
D^+_{0,0}  &\stackrel{E_2}{\longleftrightarrow} D_{0, 1}^+ \\ 
D^+_{0,0}  &\stackrel{E_j}{\longleftrightarrow} D_{0, 1}^+, \,\,\, 3 \leq j \leq u\\ 
D^+_{0,0}  &\stackrel{E_k}{\longleftrightarrow} D_{0, -1}^+, \,\,\, v \leq k  \leq -1\\
\end{align*}

\vskip5pt

If $D$  is the rightmost disk in the sequence  (i.e. $ t  = 0$)  then replace

$$D^+_{0,0}  \stackrel{E_2}{\longleftrightarrow} D_{0, 1}^+ $$

by 

$$D^+_{0,0}  \stackrel{E_2}{\longleftrightarrow} D_{0, 1}^+, D_{1, 1}^+ $$

\begin{figure}[ht]
\centering
\includegraphics[width=4in]{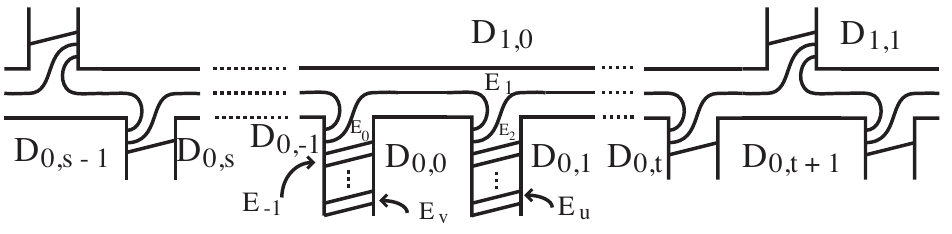}
\caption{The case  where $tv(D) = 0$, discussed in A3(a).}
\label{A3a}
\end{figure}

\item If $D$ is in the top row (see Figure \ref{A3b}) then the region $E_1$ connects $D_{0,0}^+$ to all 
disks  in the layer as well as to the back of the leftmost disk: 
 
 $$D^+_{0,0}  \stackrel{E_1}{\longleftrightarrow} D_{0, s}^-, D_{0, s}^+, \dots,  D_{0, t}^+,  $$
 
 where $ s  < 0 < t$ and $ t - s +1 = c$. The other connections are the same as in case $(a)$.

 \begin{figure}[ht]
\centering
\includegraphics[width=4in]{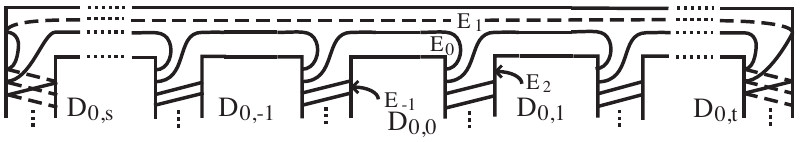}
\caption{The case  where $tv(D) = 0$, discussed in A3(b).}
\label{A3b}
\end{figure}
\end{enumerate}
\newcounter{savecount}
\setcounter{savecount}{\value{enumi}}
\end{enumerate}

\subsection*{Edge disks}

We now consider disks which are either leftmost or rightmost in their layer.

\begin{enumerate}
\setcounter{enumi}{\value{savecount}}

\item[(A4)] Let $D_{0,0}$ be the right disk in the top layer.  
This case is  as  in case A3(b) except that we set $t = 0$.  The
region $E_2$ now connects  $D^+_{0,0}$ to the  back side   $D_{0,i}^-$
of all  disks  in  the top layer, and we replace  

$$D^+_{0,0}  \stackrel{E_j}{\longleftrightarrow} D_{0, 1}^+, \,\,\, 3 \leq j \leq u$$ 

by

$$D^+_{0,0}  \stackrel{E_j}{\longleftrightarrow} D_{0, 0}^-, \,\,\, 3 \leq j \leq u$$

\item[(A5)] Let $D_{0,0}$ be the  left disk in a top layer. 

This case is as in  case A3(b) except that we set $s = 0$. Now $E_{k}, \, v \le k \leq 0$, connects  $D^+_{0,0}$ to  $D^-_{0,0}$.

\vskip7pt

 \noindent The following cases do not occur in top or bottom layers.

\vskip7pt

\item[(A6)] Let $D_{0,0}$ be the rightmost disk in a layer protruding to the left (see Figure \ref{Rightmost in left protruding}). 

\begin{align*}
D^+_{0,0}  &\stackrel{E_0}{\longleftrightarrow} D_{0, -1}^+,D_{1, 0}^+ \\
D^+_{0,0}  &\stackrel{E_1}{\longleftrightarrow} D_{1,0}^+, D_{1,0}^-\\
D^+_{0,0}  &\stackrel{E_2}{\longleftrightarrow} D_{1, 0}^-, D_{0,s}^-, \dots, D_{0,0}^-\\ 
D^+_{0,0}  &\stackrel{E_j}{\longleftrightarrow} D_{0, 0}^-, \,\,\, 3 \le j \le u\\ 
D^+_{0,0}  &\stackrel{E_k}{\longleftrightarrow} D_{0, -1}^+, \,\,\, v \le k \leq -1\\
\end{align*}

\noindent where $D_{0,s}$ is the leftmost disk in the layer.

\vskip7pt

\begin{figure}[ht]
\centering
\includegraphics[width=4in]{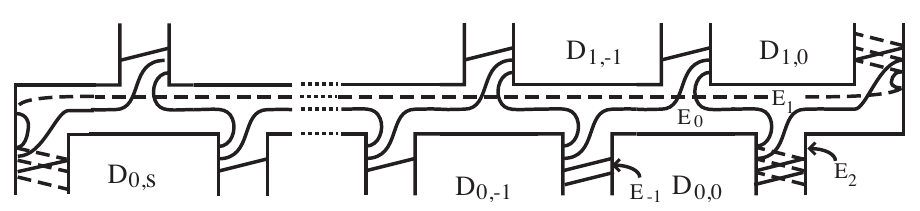}
\caption{Rightmost disk in layer protruding to the left as discussed in A6.}
\label{Rightmost in left protruding}
\end{figure}

\item[(A7)] Let $D_{0,0}$ be the leftmost disk in a layer protruding to the left (see Figure \ref{Leftmost in left protruding}). 

\begin{align*}
D^+_{0,0}  &\stackrel{E_0}{\longleftrightarrow} D_{0,0}^-,D_{1, 0}^+ \\
D^+_{0,0}  &\stackrel{E_1}{\longleftrightarrow} D_{1,0}^+, D_{1,1}^+\\
D^+_{0,0}  &\stackrel{E_2}{\longleftrightarrow} D_{0,1}^+, D_{1, 1}^+\\ 
D^+_{0,0}  &\stackrel{E_j}{\longleftrightarrow} D_{0, 1}^+, \,\,\, 3 \le j \le u\\ 
D^+_{0,0}  &\stackrel{E_k}{\longleftrightarrow} D_{0,0}^-, \,\,\, v \le k \leq -1\\
\end{align*}

Note that  $tv(D_{0,1}) = 1$ since we have a nice flypeable trellis.
\vskip7pt
\begin{figure}[ht]
\centering
\includegraphics[width=2.5in]{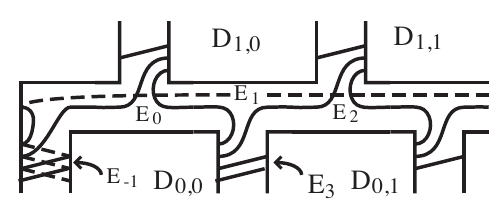}
\caption{Leftmost disk in layer protruding to the left as discussed in A7.}
\label{Leftmost in left protruding}
\end{figure}

\item[(A8)] Let $D_{0,0}$ be the rightmost disk in a layer protruding to the right (see Figure \ref{A8}). 

\begin{align*}
D^+_{0,0}  &\stackrel{E_0}{\longleftrightarrow} D_{1, -1}^-,D_{0,-1}^+ \\
D^+_{0,0}  &\stackrel{E_1}{\longleftrightarrow} D_{1,t}^+, D_{1,t}^-, \dots, D_{1,-1}^-\\
D^+_{0,0}  &\stackrel{E_2}{\longleftrightarrow} D_{0, s}^-, \dots,  D_{0,0}^-, D_{1,t}^+\\ 
D^+_{0,0}  &\stackrel{E_j}{\longleftrightarrow} D_{0, 0}^-, \,\,\,3 \le j \le u\\ 
D^+_{0,0}  &\stackrel{E_k}{\longleftrightarrow} D_{0, -1}^+, \,\,\, v \le k \leq -1\\
\end{align*}
where $D_{0,s}$ and $D_{1,t}$ are the leftmost disks in their layers.

\begin{figure}[ht]
\centering
\includegraphics[width=4in]{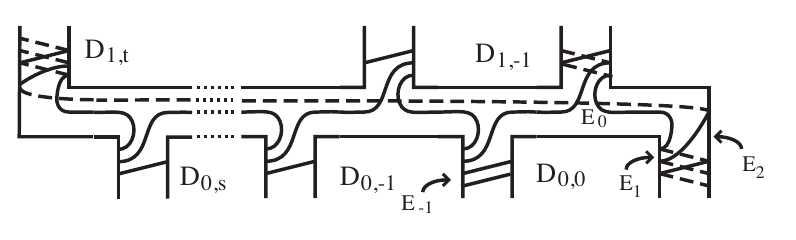}
\caption{Rightmost disk in layer protruding to the right as discussed in A8.}
\label{A8}
\end{figure}

\item[(A9)] Let $D_{0,0}$ be the leftmost disk in a layer protruding
  to the right (see Figure \ref{A9}).

Note that  $tv(D_{0,1}) = 1$ as we have a nice flypeable trellis.

\begin{align*}
D^+_{0,0}  &\stackrel{E_0}{\longleftrightarrow} D_{0,0}^-,D_{1, 0}^+ \\
D^+_{0,0}  &\stackrel{E_1}{\longleftrightarrow} D_{1,0}^+,D_{1,1}^+\\
D^+_{0,0}  &\stackrel{E_2}{\longleftrightarrow} D_{0, 1}^+, D_{1,1}^+\\ 
D^+_{0,0}  &\stackrel{E_j}{\longleftrightarrow} D_{0, 1}^+, \,\,\,3 \le  j \le u\\ 
D^+_{0,0}  &\stackrel{E_k}{\longleftrightarrow} D_{0,0}^-, \,\,\,  v \le k \leq -1\\
\end{align*}

\begin{figure}[ht]
\centering
\includegraphics[width=3.in]{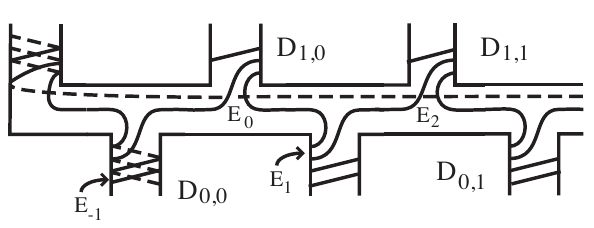}
\caption{Leftmost disk in layer protruding to the right as discussed in A9.}
\label{A9}
\end{figure}

\end{enumerate}

\vskip7pt

Now consider disks of type $D^-$ (see Figure \ref{Backconfiguration}):

\vskip7pt

\begin{figure}[ht]
\centering
\includegraphics[width=3.0in]{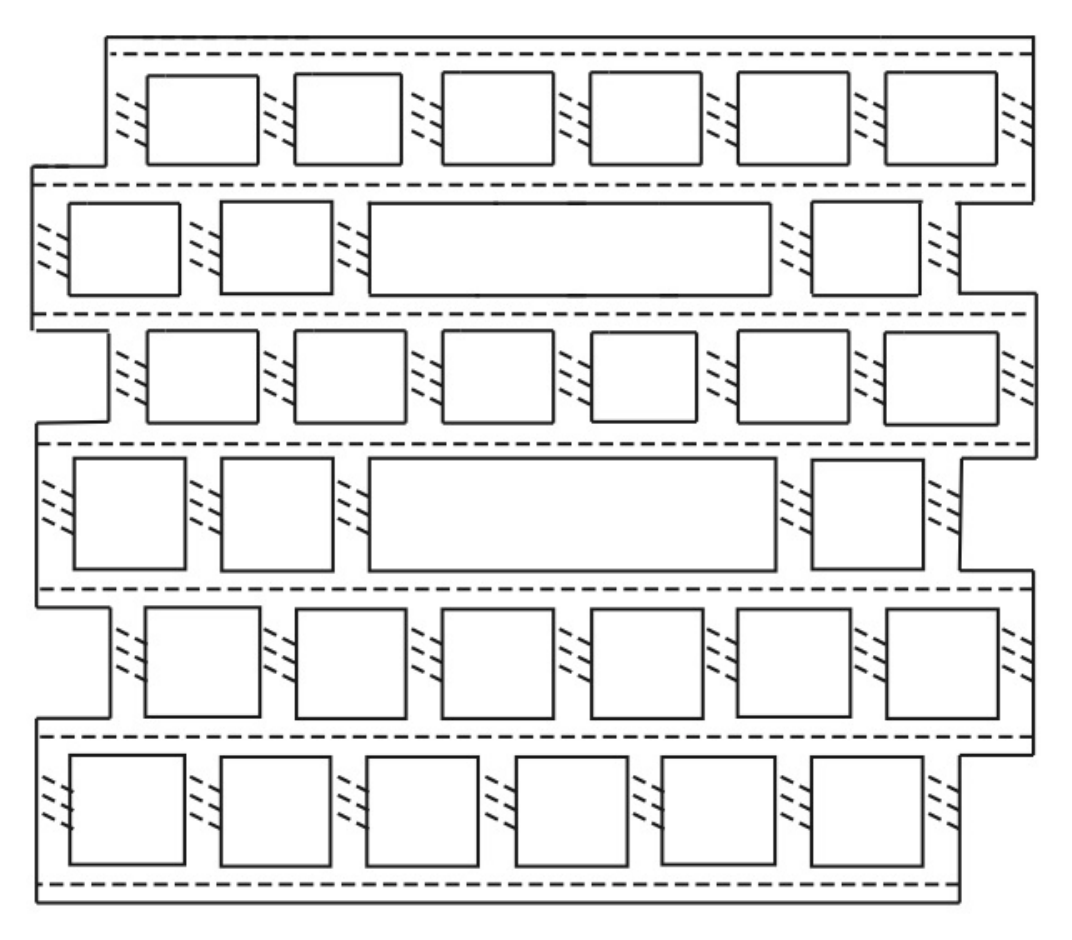}
\caption{Configuration on the back of $H_1$.   The sub-arcs of the knot are indicated by the dashed arcs.}
\label{Backconfiguration}
\end{figure}

In the back of each layer there are two ``long'' regions $F$ and $F'$ which meet every back disk in its top
or bottom edge respectively. In addition, in every interior column there is a sequence of regions which meet only 
the two adjacent disks to the column. The regions that meet the leftmost or rightmost columns (including $F$ 
and $F'$) give connections from back disks to front disks in the same or adjacent layers. These connections 
were given 
in the discussion of the front disks.

\vskip7pt

The cases described above, together with their $180^\circ$ rotations, give all possible connections between disks. 
For example, the connections along the bottoms of the disks in cases (A4) and (A5) are obtained as rotations of 
the connections in cases (A6--A9).

\vskip7pt

It can now be checked that any time two $\EE$ regions connect disks which are not in the same layer then these
regions are adjacent along a single arc of $K \ssm \DD$. Here are some examples of this analysis:

In case (A1)(a) the only connections between $D^+_{0,0}$ and disks in a different row are  

\begin{align*}
D^+_{0,0}  &\stackrel{E_0}{\longleftrightarrow} D_{1,0}^+\\
D^+_{0,0}  &\stackrel{E_1}{\longleftrightarrow} D_{1,0}^+  \,\, \,and \\
D^+_{0,0}  &\stackrel{E_1}{\longleftrightarrow} D^+_{1,1}\\ 
\end{align*}

Note that $E_0$ and $E_1$ are adjacent along a single arc and  hence the rectangle  determined by the first two 
lines is trivial.

\vskip5pt

In case (A2) the only connections between $D^+_{0,0}$ and disks in a
different row are:  

\begin{align*}
D^+_{0,0}  &\stackrel{E_0}{\longleftrightarrow} D_{1,0}^+,\ldots,D_{1,s-1}^+\\
D^+_{0,0}  &\stackrel{E_0}{\longleftrightarrow} D_{1,s}^+\\
D^+_{0,0}  &\stackrel{E_1}{\longleftrightarrow} D_{1,s}^+ \\
D^+_{0,0}  &\stackrel{E_1}{\longleftrightarrow} D^+_{1,s+1}\\ 
D^+_{0,0}  &\stackrel{E_2}{\longleftrightarrow} D^+_{1,s+1}\\ 
\end{align*}

Here the connections in the first line do not belong to any
rectangle. The second and third line define a trivial rectangle, and
so do the third and fourth. 

\vskip5pt

In case (A7) the only connections between $D^+_{0,0}$ and disks in a different row are:  

\begin{align*}
D^+_{0,0}  &\stackrel{E_0}{\longleftrightarrow} D_{1,0}^+\\
D^+_{0,0}  &\stackrel{E_1}{\longleftrightarrow} D_{1,0}^+ \\
D^+_{0,0}  &\stackrel{E_1}{\longleftrightarrow} D^+_{1,1}\\ 
D^+_{0,0}  &\stackrel{E_2}{\longleftrightarrow} D^+_{1,1}\\ 
\end{align*}

\vskip5pt

Here the first and second lines define a trivial rectangle as do the
third and fourth.

Finally, let $D^-$ be a back disk in the middle of a layer protruding
to the left. It is connected to the the rightmost front disk of the
layer below using the connections in  (A8) second line, and to the
rightmost back disk of the layer above using the connections in  (A6)
third line. None of these connections is part of a rectangle. 

\vskip5pt

The remaining cases are similar (for the global picture it is helpful
to  consult  Figure \ref{KnotcarriedbyNFT}), and an inspection of them completes
the proof of the lemma.

\end{proof}

\vskip5pt

We now compile a list of the nontrivial rectangles. First we list nontrivial rectangles from front disks to front disks.

\vskip5pt

\begin{enumerate}

\item[(B1)] Disks which share a vertical column $e$  (see Figure \ref{B1}). The disks will be numbered $D_{0,0}$ 
and $D_{0,1}$. The $\EE$ regions will be numbered clockwise around $D_{0,0}$ starting from the top. 

\vskip5pt

\begin{enumerate}

\item $tv(D_{0,1}) > 0$,   $bv(D_{0,0}) > 0$: 

We obtain a rectangle for  ($D^+_{0,0}, D^+_{0,1}, E_i, E_j$), 

where $2 \leq i < j \leq a(e)$.

\vskip5pt

\item $tv(D_{0,1}) = 0$,   $bv(D_{0,0}) = 0$: 

We obtain a rectangle for ($D^+_{0,0}, D^+_{0,1}, E_i, E_j$), 

where $1 \leq i < j \leq a(e) + 1$.

\vskip5pt

\item   $tv(D_{0,1}) =  0$,   $bv(D_{0,0}) > 0$:

We obtain a rectangle for  ($D^+_{0,0}, D^+_{0,1}, E_i, E_j$), 

where $1 \leq i < j \leq a(e)$.

\vskip5pt

\item $tv(D_{0,1}) > 0$,   $bv(D_{0,0}) = 0$: 

This case is obtained from the previous  one by a $180^\circ$ rotation.

\end{enumerate}

\vskip10pt

 \noindent In all preceding cases the rectangles are nontrivial  when $j - i \geq 2$.
 
 \begin{figure}[ht]
\centering
\includegraphics[width=4.6in]{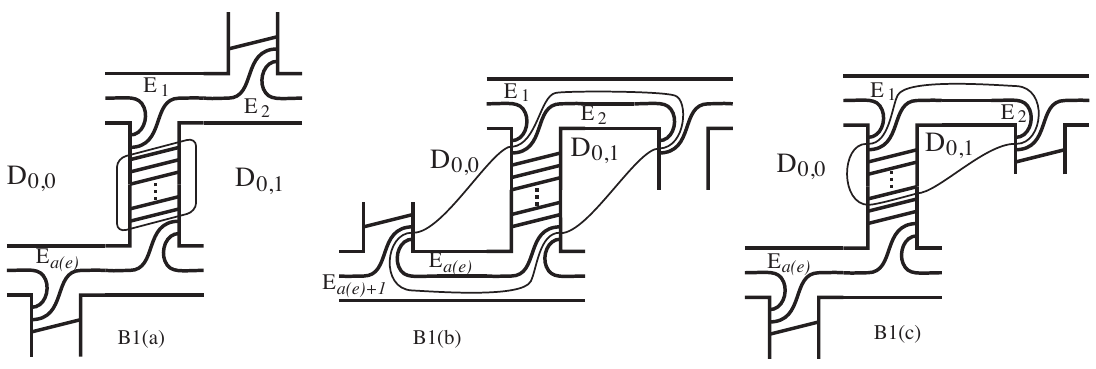}
\caption{Case B1. For each subcase one example rectangle is indicated by a simple closed curve tracing out its boundary. 
In (a), $(i,j) = (2,a(e))$. In (b), $(i,j)=(1,a(e)+1)$. In (c), $(i,j)=(1,a(e)-2)$.}
\label{B1}
\end{figure}

\vskip5pt
 
 \item[(B2)] Consider a sequence of disks in a layer which is contained in a flype box or is  in a top or bottom 
 layer which is adjacent to a flype box. In such cases we obtain a sequence   $D_{0,0}, \dots, D_{0,p}$ so that 
 
 $$tv(D_{0,k}) = 0, \, 1 \leq k \leq p,$$
 
$$bv(D_{0,k}) = 0, \, 0 \leq k \leq p - 1$$

\vskip5pt

We obtain a nontrivial rectangle ($D^+_{0,0}, D^+_{0,p}, E_1, E_{-1}$). Note that the case $p = 1$ already 
appears in case (B1).

 \begin{figure}[ht]
\centering
\includegraphics[width=3.3in]{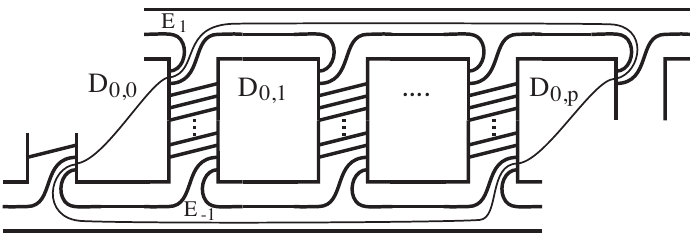}
\caption{Case B2.}
\label{B2}
\end{figure}

\setcounter{savecount}{\value{enumi}}
\end{enumerate}

We now consider rectangles which connect back disks to back disks. The back of every  layer has the same structure, 
which can be  seen in Figure \ref{Backconfiguration}: There are two ``long'' $\EE$ regions  denoted by $F, F'$, which 
meet every disk on its  top and bottom edge respectively. Given a column $e$  there is a sequence of at least two regions 
which connect the disks adjacent to $e$,  which we number $F_1, \dots, F_{a(e) - 1}$ from top to bottom.

\vskip5pt

\begin{enumerate}
\setcounter{enumi}{\value{savecount}}

\item[(B3)] For any two  $D_{0,p}, D_{0,q}$ in a layer there is a rectangle ($D^-_{0,p}, D^-_{0,q}, F, F' $).

\vskip5pt

\item[(B4)] Let $e$ be the column between $D_{0,p}$ and $D_{0,p +1}$, and $F_1,\ldots,F_{a(e)-1}$ the 
associated regions. Moreover write $F _0 = F$ and $F_{a(e)} = F'$.  These determine rectangles 
$$(D^-_{0,p}, D^-_{0,p + 1}, F_i, F_j), \,\,  0 \le i < j \le a(e)  $$  which are nontrivial when 
$j - i \ge 2$.

\vskip5pt

\setcounter{savecount}{\value{enumi}}
\end{enumerate}

We now consider rectangles which connect front disks to back disks. These occur at the right and left edges of the trellis.  

In  $(B5 - B10)$ let $D$ be a rightmost or leftmost disk in 
a layer, and let $e$ be the column  to its right or left respectively. We number the regions meeting $D^+$ clockwise 
around $D$, with $E_0$ meeting the  top edge. We then have the following rectangles, all of which can be seen in 
 Figures \ref{KnotcarriedbyNFT} and  \ref{Backconfiguration}:
\vskip5pt

\begin{enumerate}
\setcounter{enumi}{\value{savecount}}
\item[(B5)] Let $D$ be a rightmost disk in a left protruding inner layer. There are  rectangles: 
 
$$(D^+, D^-, E_i,  E_j), \,\,  2 \le i < j \le a(e)$$

\vskip5pt

\noindent  Note that $D^+$ connects to the back of the first disk in the layer through only one region, 
hence there  is no corresponding rectangle.

\item[(B6)] Let $D$ be a rightmost disk in a right protruding inner layer. There  are rectangles: 
 
$$(D^+, D^-, E_i,  E_j), \,\,  2 \le i < j \le a(e)$$

\vskip5pt

\item[(B7)] Let $D$ be a rightmost disk in a right protruding bottom layer. There are rectangles: 
 
$$(D^+, D^-, E_i,  E_j), \,\,  2 \le i < j \le a(e) +1$$

\vskip25pt

\item[(B8)] Let $D$ be a leftmost disk in a right protruding bottom layer. There are rectangles: 
 
$$(D^+, D^-, E_i,  E_j), \,\,  -a(e) +2 \le j < i \le 0$$

\vskip5pt

\item[(B9)] Let $D$ be a leftmost disk in a right protruding top layer. There are rectangles: 
 
$$(D^+, D^-, E_i,  E_j), \,\,  -a(e) + 2 \, \le j < i \le 1$$

\vskip4pt

\item[(B10)] Let $D$ be a rightmost disk in a right protruding top layer. There are rectangles: 
 
$$(D^+, D^-, E_i,  E_j), \,\,  2 \le i < j \le a(e)$$

\vskip4pt

\end{enumerate}

All remaining cases are obtained from the above by a $180^\circ$ rotation of the plane $P$.
The rectangles are non trivial when $j - i \geq 2$.

\vskip4pt

\begin{lemma}\label {enumerate rectangles}  All nontrivial rectangles are described in cases 
$(B1) - (B10)$. 

\end{lemma}

\vskip4pt

\begin{proof} 
By Lemma \ref{single layer}, we need to consider only rectangles between disks contained in a single horizontal 
layer. The proof is then a case by case inspection, using the same data and techniques as the proof of 
Lemma \ref{single layer}.

\end{proof}

\vskip7pt

Two rectangles $R = (D_1^{e_1}, D_2^{e_2}, E_1, E_2)$, and 
$R' = (D_3^{e_3}, D_4^{e_4}, E_3, E_4)$, where $e_i \in \{\pm 1\}, i = 1 \dots, 4$, will be 
called {\it adjacent along $D_2$} if the following holds:

\begin{enumerate}

\item $D_2 = D_3$ and $e_2 = - e_3$ and

\vskip4pt

\item The arcs of intersection $E_1 \cap D_2, E_2 \cap D_2$ are equal to the arcs of intersection
$E_3 \cap D_3, E_4 \cap D_3$.

\end{enumerate}
This condition captures the combinatorial aspects of an adjacency of
rectangles in the annulus $A$.

\vskip4pt

If after renumbering $R$ and $R'$ are adjacent along one of the disks, we say that they are {\it adjacent}. 

\begin{lemma} \label{no mixing} There are no adjacencies between trivial and non trivial rectangles.

\end{lemma}

\vskip7pt

\begin{proof} An inspection of cases (B1 - B10) shows that given a non trivial rectangle 
$R = (D_1^{e_1}, D_2^{e_2}, E_1, E_2)$, the arcs of intersection $E_1\cap D_i$ and $E_2\cap D_i$ 
for each $i=1,2$ are separated by at least two points of $K\cap \partial D_i$. Since for a trivial rectangle 
these arcs are always separated by just one such point, there can be no adjacencies between trivial and 
nontrivial rectangles. 

\end{proof}

\vskip7pt

The {\it adjacency graph of rectangles} will be the graph whose vertices are rectangles, where we place and 
edge between $R$ and $R'$, labeled by $D$, whenever $R$ and $R'$ are adjacent along $D$. Formally speaking  
a pair rectangles might have distinct adjacencies labeled by the same disk. However, the arguments in 
Lemma \ref{rectangle adjacencies} show that  this never occurs in our setting.

\vskip7pt

\begin{lemma}\label{rectangle adjacencies} In the adjacency graph of rectangles, every cycle contains 
only trivial rectangles.

\end{lemma}

\vskip7pt

\begin{proof} 
By Lemma \ref{no mixing}, if a cycle contains any trivial rectangle then it contains only trivial rectangles. 
Hence it suffices to restrict to the subgraph of nontrivial rectangles and show that it contains no cycles. 

Wrapping around each vertical column there are (several) sequences of adjacent rectangles. Consider for 
example case (B1)(a). A front rectangle indexed by $(i,j), i \geq 3, $ is adjacent to a back rectangle from 
case (B4) indexed by $(i - 2, j - 2)$. If $i - 2 \geq 2$   then this rectangle is adjacent to a front rectangle 
indexed by $(i - 2, j - 2)$. If  $i - 2 \leq 1$ there are no further adjacencies. Hence any such chain terminates 
in a rectangle that has no further adjacencies and thus is not part of a cycle.

Similar arguments apply to the rest of case (B1), with one proviso: If $tv(D_{0,1})=0$ and a back rectangle is 
indexed by $(1,j)$, then it is adjacent to one further front rectangle indexed by $(1,j)$ (see Figure \ref{B1} case (c)), 
but that rectangle meets $D_{0,1}$ in two vertical arcs on opposite edges. Any  rectangle involving $D_{0,1}^-$ 
meets it either in vertical arcs on the same edge of $\partial D_{0,1}$ (cases (B4 - B10)), one vertical arc and one 
horizontal edge (case (B4)), or in horizontal edges (case (B3)). Hence the chain terminates at this point. 

A front rectangle in case (B2) again meets its disks along vertical arcs on opposite edges, and so is not adjacent 
to any rectangle. 

In cases (B5 - B10), a similar analysis as case (B1) holds. Note that in these cases rectangles are not divided 
into ``front'' and ``back'', rather each rectangle wraps around from front to back.

Every adjacency of a back rectangle must be to a rectangle of type already discussed, hence back rectangles 
cannot be a part of a cycle either.

\end{proof}

Every essential annulus  $A \subset (H_2, \NN(K))$ determines a cycle of rectangles in the adjacency graph, 
and Lemma \ref{rectangle adjacencies} implies that all these rectangles are trivial. Hence the annulus is parallel 
to the knot. This finishes the proof of Proposition \ref{boundaryparallelannuli}

\end{proof}

 Lemma \ref{suffices to prove} together with Proposition
 \ref{boundaryparallelannuli} and Thurston's Haken geometrization theorem 
 complete the proof of  Theorem \ref{hyperbolic}.
 
\end{proof}

\vskip10pt

\subsection {Ruling out $I$-bundles}

In order to apply Proposition \ref{mainproposition} to the proof of Theorem \ref{maintheorem} we need to further  
show that $(H^r_2 ,\bar\NN_{\Sigma^r}(K^r))$ is not an $I$-bundle. We do this in the following lemma:

\vskip5pt

\begin{proposition}\label{flypedtrellis} For each $r \in \mathbb{Z}$
  the pared manifold $(H^r_2 ,\bar \NN_{\Sigma^r}(K^r))$ is not an $I$-bundle. 

\end{proposition}

\begin{proof} For $r = 0$ this is a consequence  of Proposition \ref{boundaryparallelannuli}. The rest of 
the proof will be  given for $r > 0$. The case $r < 0$ will follow from the usual $180^{\circ}$ rotation.

In this case it is easy to see that $(H^r_2,  \NN_{\Sigma^r}(K^r))$ does in fact contain  essential annuli. 
In particular  Lemma \ref{no mixing} fails because now there are columns with $a(e) = 1$ and this allows
adjacencies between trivial and nontrivial rectangles.

The idea is to prove that there is a union of (one or two) essential annuli in  $(H^r_2,  \NN_{\Sigma^r}(K^r))$ 
which separates $(H^r_2,  \NN_{\Sigma^r}(K^r))$ into two components, one of which contains no essential annuli. 
This is impossible in an $I$-bundle, since an $I$-bundle does not have a non-trivial JSJ decomposition  
(see \cite{jaco-shalen}) hence the proposition follows.

Consider the rectangular closed curve labeled $\tau$ in Figure \ref{BFS}. Let $\tau^+$ ($\tau^-$) denote 
the curve in  the front (back) of $\Sigma$ lying in front of  (back of) $\tau$. The curves $\tau^+$, $\tau^-$ 
bound disks denoted by $\Delta_F$  and $\Delta_B$ in $H_2^r$, whose projection to $P$ is the disk 
$\Delta$ bounded by $\tau$. The points between $\Delta_F$  and $\Delta_B$ which  project to $\Delta$ 
form a $3$-ball denoted by $B^{fl}$ called the {\it flype box}.

\begin{figure}[ht]
\centering
\includegraphics[width=4.8in]{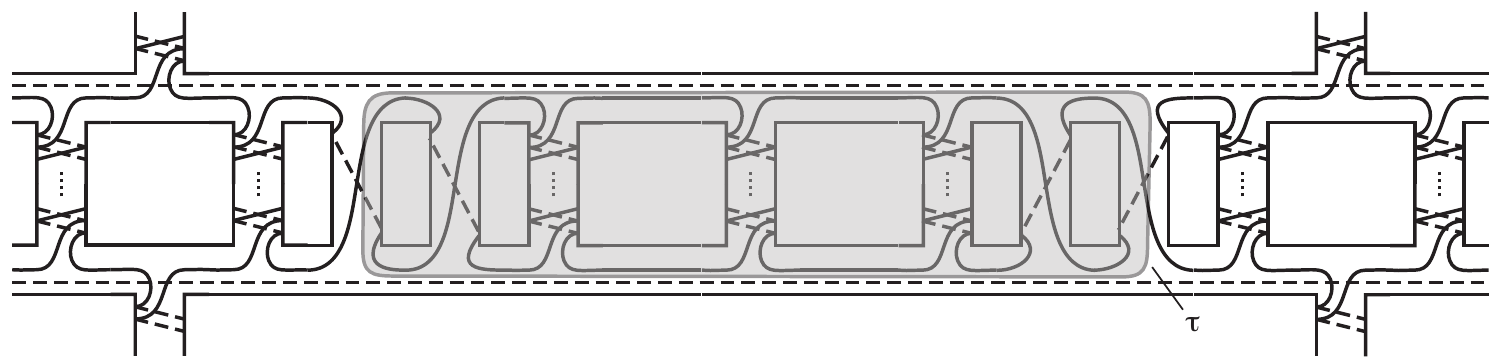}
\caption{Constructing  the flype box for $r = 2$. The disk $\Delta$ is shaded and $\tau$ is its boundary.}
 \label{BFS}
\end{figure}

Let $Y = B^{fl} \cup \NN_{S^3}(K)$, where we  choose the neighborhood $\NN(K)$ so small so that $Y$ 
is a genus  two handlebody. Note that $\AAA= \partial Y \cap H_2^r$ is a union of one or two annuli, depending 
on $K$. Furthermore set $\widehat{H}^r_1 = H_1^r \cup Y,\,\, \widehat{H}^r_2 = S^3 \ssm int(\widehat{H}^r_1)$. 
Note also that $\AAA$ separates  $H_2^r$ into two components one of
which has closure $\widehat{H}^r_2$. We now show 
that the pared manifold $(\widehat{H}^r_2, \AAA)$ contains no essential  annuli.

The core of $\AAA$ is a link $\widehat{K}^r$ on $\widehat{\Sigma}^r = \partial \widehat{H}^r_1 = \partial \widehat{H}^r_2$. 
The link  $\widehat{K}^r$ is carried by a new trellis $\widehat{\TT}^r$ where $B^{fl}$ is replaced by a single vertical column 
(see Figure \ref{TRFB}). The projection into $P$ of $K^r$ outside $B^{fl}$ is equal to the projection of
$\widehat{K}^r$ outside this column.

\begin{figure}[ht]
\centering
\includegraphics[width=4.5in]{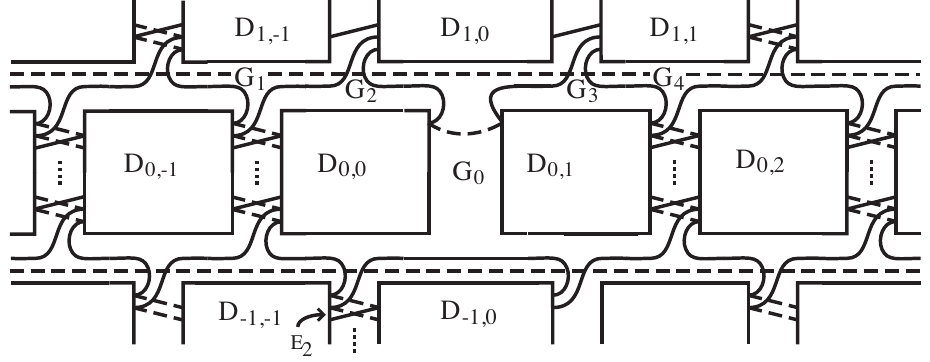}
\caption{The new trellis $\widehat\TT^r$, obtained by amalgamating the
  flype box and the columns it meets into one new column. The new
  column is here between $D_{0,0}$ and $D_{0,1}$. }
 \label{TRFB}
\end{figure}

We apply the same techniques as in the proof of Lemma \ref{single
  layer}. Note that there are four front 
$\EE$ regions $G_0, G_1, G_2, G_3$ whose configurations are somewhat different. The other front $\EE$ regions 
stay the same. The back $\EE$ regions stay the same but note that in
the back of the new column there are no ``small''  
regions connecting just  $D^-_{0,0}$ and $D^-_{0,1}$, because there is
only one arc in that column.

Note by inspection that all $\EE$ regions are disks and that every region intersects each disk side in 
$\widehat{\DD}$ in at most a single arc. In other words this recovers
conditions (1-3) coming from Lemma 2.2 of
\cite{lustig-moriah:highgenus}, as used in Step (1) of Proposition \ref{boundaryparallelannuli}.
Therefore we can  apply the same proof as in Step (1)
to conclude that any  annulus in $(\widehat{H}^r_2, \AAA)$, after
suitable isotopy, is 
decomposed into a cycle of rectangles. 

An inspection of the diagram yields the following new nontrivial 
rectangles. (There are also rectangles that have appeared in previous
cases, and which are not listed below.)

\vskip40pt

\subsection* {Front rectangles}

\begin{enumerate}

\item [(C1)]  $(D^+_{-1,-1}, D^+_{-1,0}, G_0, E_i),  \,\, i \geq 3$. The regions  $E_i$ are in the column between 
$D_{-1,-1}$ and $D_{-1,0}$, as indicated in Figure \ref{TRFB}.

\vskip7pt

\item [(C2)]  $(D^+_{0,1}, D^+_{1,1}, G_0, G_4)$ 

\end{enumerate}

\subsection* {Back rectangles}  We denote the ``long" back regions in the layer of $D_{0,0}$ by $F$ and $F'$, 
as in case (B3). We also enumerate the ``small" back regions in the column between $D_{0,0}$ and $D_{0,-1}$ as $
F^{-1}_1,\ldots,F^{-1}_q$, and similarly the ``small" back regions in the column between $D_{0,1}$ and $D_{0,2}$ as
$F^1_1,\ldots,F^1_p$. We then obtain:

\vskip7pt

\begin{enumerate}

\item [(C3)]  $(D^-_{0,1}, D^-_{0,i}, F, F'),  \,\, i \neq 0$ 

\vskip7pt

\item [(C4)]  $(D^-_{0,0}, D^-_{0,i}, F, F'),  \,\, i \neq 1$ 

\end{enumerate}

The last two cases are of a type already discussed in (B4), but we mention them here because we must 
analyze their potential interaction with the new rectangles.

\begin{enumerate}

\item [(C5)] (a)  $(D^-_{0,1}, D^-_{0,2}, F, F^1_j),  \,\, 2\le j \le p$  and \\
(b) $(D^-_{0,1}, D^-_{0,2}, F', F^1_j),  \,\,1 \le j \le p-1$.

\vskip7pt

\item [(C6)]  (a) $(D^-_{0,0}, D^-_{0,-1}, F, F^{-1}_j),  \,\, 2\le j \le q$ and\\
(b) $(D^-_{0,0}, D^-_{0,-1}, F', F^{-1}_j),  \,\, 1 \le j \le q-1$.

\end{enumerate}

\vskip5pt

Note that the $\EE$ region $G_0$ connects a large number of disks namely 
$D^+_{-1,-1}, D^+_{-1, 0},  D^+_{1,0}, D^+_{1,1}, D^+_{0,0}, D^+_{0,1}$ and $ D^+_{0,2}$.
However only a few of these participate in nontrivial rectangles as indicated in (C1) and (C2).

The rectangles in case (C1) are not adjacent to any rectangle along $D_{-1,0}$ using the same argument 
as in case (B1) in the proof of Lemma \ref{rectangle adjacencies}. The rectangle in (C2) is adjacent along 
$D_{0,1}$  to a rectangle in case (C5)(b), for $ j = 1$. That rectangle has no further adjacencies and hence 
cannot participate in a cycle.

In case (C3) the rectangles are adjacent on one side to a trivial rectangle. However on the other side they have 
no further adjacencies since there are no front rectangles meeting opposite horizontal edges of a disk (see the 
analysis of (B3) in Lemma \ref{rectangle adjacencies}). Case (C4) is handled similarly.

Case (C5) (a) The rectangles in this case have no adjacencies along $D_{0,1}$. The rectangles in  case (b) have 
no adjacencies along $D_{0,2}$.

Case (C6) (a) The rectangles there have no adjacencies along $D_{0, -1}$. In case (b) the rectangles have
no adjacencies along $D_{0,0}$.

The  cases above together with the analysis in Lemma \ref{rectangle adjacencies} show that non-trivial rectangles
cannot participate in cycles. This proves that there are essential annuli in $(\widehat{H}^r_2, \AAA)$ and this 
completes the  proof of the proposition.

\end{proof}

\section{Finishing the proof}\label{main proof}

\vskip10pt

We can now assemble the previous results to produce a 
sequence of primitive stable discrete faithful representations with
rank going to infinity which 
converges geometrically to a knot complement. 

\begin{proof}[Proof of Theorem \ref{maintheorem}]
Let $K \subset S^3$ be a knot carried by a nice flypeable trellis  $\mathcal{T}$ and satisfying the conditions of 
Theorem \ref{main-incomp}. The manifold  $ M_\infty = S^3 -  \nbhd(K)$  is hyperbolic by Proposition
\ref{hyperbolic},  so we have a discrete faithful representation   
$\eta: \pi_1(M_\infty) \to PSL_2(\mathbb{C})$.

For each $r\in\mathbb{N}$, consider the decomposition of $M_\infty$ along $\Sigma^r \ssm \nbhd(K^r)$ into 
two handlebodies 
$$V^r = H_1^r \ssm \NN(K^r)$$
and
$$W^r = H_2^r \ssm \NN(K^r).$$
 Let $i^r_* : \pi_1(V^r) \to \pi_1(M_\infty)$ be induced by the
inclusion map. Recall that $\pi_1(V^r) = F_{n_r}$, where $n_r = n_0 + 2r$. 

\begin{equation}\label{rho and eta}
\xymatrix{
                     & \pi_1(M_\infty) \ar[dr]^\eta \ar[dd]^(.7){q_m} &       \\
\pi_1(V^r) \ar[ur]^{i^r_*} \ar@{-->}[rr]^(.3){\rho^r_m}          &                  & PSL_2(\C)   \\
                     & \pi_1(K_{\Sigma^r}(1/m)) \ar[ur]^{\eta_m}  & \\
}
\end{equation}
\medskip

We let $K_{\Sigma^r}(p/q)$ denote the $p/q$ Dehn filling of $K^r$ with respect to the framing of $\Sigma^r$ 
as in  Subsection \ref{HSKT}, where we have abbreviated $K_{\Sigma^r}=K^r_{\Sigma^r}$. 
For each $m\in \Z$, let $q_m : \pi_1(M_\infty) \to \pi_1(K_{\Sigma^r}(1/m))$ be the quotient map induced by surgery.  

By Thurston's Dehn filling theorem, for large enough $|m|$ the manifolds $K_{\Sigma^r}(1/m)$  are hyperbolic, 
and there are discrete faithful representations  $\eta_m:\pi_1(K_{\Sigma^r}(1/m)) \to PSL_2(\C)$ such that the 
representations $\eta_m\circ q_m $ converge to $\eta$.   Moreover the quotient manifolds converge geometrically 
to $M_\infty$. 

Because the surgered manifold $K_{\Sigma^r}(1/m)$ is obtained by an $m$-fold Dehn twist on $K^r$, the images of 
$V^r$ and $W^r$ determine a Heegaard splitting for this manifold, and in particular the map $q_m \circ i^r_*$ is surjective.

Letting $m\to \infty$, the representations $\rho^r_m = \eta_m\circ q_m \circ i^r_*$ converge 
to $\rho^r_\infty = \eta \circ i^r_*$.

The representation $\rho^r_\infty$ satisfies the hypotheses of Proposition \ref{mainproposition}:  Hypothesis (1) 
(a cut point free Whitehead graph) follows from Lemma \ref{WHG}; hypothesis (2)  (incompressibility of 
$\Sigma^r \ssm  K^r$)  follows from Theorem \ref{main-incomp}; and hypothesis (3) (the pared manifold 
$(W^r, \bar\NN(K^r) \cap \partial W^r)$ is not an $I$ bundle) follows from Proposition \ref{flypedtrellis}. We conclude, 
by  Proposition \ref{mainproposition}, that   $\rho^r_\infty$ is primitive stable.

Since the primitive stable set $PS(F_{n_r})$ is open (see Minsky \cite{minsky:primitive}), for each  
$r$ there exists $m_r$ such  that $ \rho^r_{m_r}$ is primitive stable as well. 
In particular  the image of $\rho^r\equiv\rho^r_{m_r}$ is the whole group
$\eta_{m_r}(\pi_1(K_{\Sigma^r}(1/m_r))$, and by choosing $m_r$ sufficiently large for each $r$, 
this sequence of groups converges geometrically to $\eta(\pi_1(M_\infty))$ as $r\to\infty$. 
This is the desired sequence of representations. 

\end{proof}

\vskip10pt

\bibliographystyle{hamsplain}
\bibliography{math}

\end{document}